\documentclass[10pt]{amsart}

\usepackage{amsmath}
\usepackage{amssymb}
\usepackage{amsfonts}
\usepackage{mathrsfs}
\usepackage{graphicx}
\usepackage{epsfig}
\usepackage[T1]{fontenc}
\numberwithin{equation}{section}       
\usepackage{hyperref}
\theoremstyle{plain}
\newtheorem{theorem}{Theorem}[section]

\newtheorem{prop}{Proposition}[section]

\newtheorem{coro}[prop]{Corollary}
\newtheorem{lemma}[prop]{Lemma}

\newtheorem*{mainthm}{Main Theorem}

\newenvironment{proofof}[1]{\medskip
\noindent{\it Proof of #1.}}{ \hfill\qed\\ }

\theoremstyle{definition}
\newtheorem{definition}[prop]{Definition}

\theoremstyle{remark}
\newtheorem{rema}[prop]{Remark}

\newtheoremstyle{citing}
  {3pt}
  {3pt}
  {\itshape}
  {}
  {\bfseries}
  {.}
  {.5em}
  {\thmnote{#3}}

\theoremstyle{citing}

\DeclareMathAlphabet{\mathpzc}{OT1}{pzc}{m}{it} 

\newcommand{\F}{\mathbb{F}}

\newcommand{\N}{\mathbb{N}}

\newcommand{\R}{\mathbb{R}}
\newcommand{\T}{\mathbb{T}}

\newcommand{\Z}{\mathbb{Z}}

\newcommand{\sE}{\mathscr{E}}

\newcommand{\sL}{\mathscr{L}}

\newcommand{\hE}{E}

\newcommand{\hI}{I}
\newcommand{\hJ}{\widehat{J}}

\newcommand{\hR}{\widehat{R}}

\newcommand{\hX}{\widehat{X}}

\newcommand{\hOmega}{\widehat{\Omega}}

\newcommand{\tR}{\widetilde{R}}

\newcommand{\teta}{\widetilde{\teta}}

\newcommand{\eps}{\varepsilon}

\newcommand{\Leb}{\textrm{Leb}}
\newcommand{\dist}{d}
\newcommand{\Ome}{\Omega_\eps^\N}
\newcommand{\Omea}{\hOmega_\eps^\N}

\newcommand{\PP}{P}

\DeclareMathOperator{\Dist}{Dist}

\newcommand{\tm}{\widetilde{m}}
\newcommand{\tr}{\widetilde{r}}
\newcommand{\bft}{\textbf{t}}
\newcommand{\bfs}{\textbf{s}}
\newcommand{\bfu}{\textbf{u}}
\newcommand{\parab}{\textrm{par}}
\newcommand{\rep}{\textrm{rep}}

\DeclareMathOperator{\BAD}{BAD}

\begin{document}
\title[Non-uniformly expanding circle map]{On stochastic stability of expanding circle maps with neutral fixed points}
\author{Weixiao Shen and Sebastian van Strien}
\address{Department of Mathematics,
National University of Singapore, 10 Lower Kent Ridge Road,
Singapore 119076}
\email{matsw@nus.edu.sg}
\address{Department of Mathematics, Queen's Gate, Imperial College, London SW7 2AZ, UK}
\email{s.van-strien@imperial.ac.uk}
\begin{abstract}It is well-known that the Manneville-Pomeau map with a parabolic fixed point
of the form $x\mapsto x+x^{1+\alpha} \mod 1$
is stochastically stable for $\alpha\ge 1$ and the limiting measure is the Dirac measure at the fixed point.
In this paper we show that if $\alpha\in (0,1)$ then it is also stochastically stable. Indeed,
the stationary measure of the random map converges strongly
to the absolutely continuous invariant measure for the deterministic system as the noise
tends to zero.
\end{abstract}
\date{\today}

\maketitle

\section{Introduction}
In this paper, we consider stochastic stability of (topologically) expanding circle maps with neutral fixed points under noise.
Let $\T=\R/\Z$ be the circle. For each $\alpha>0$,
let $\sE_{\alpha}$ denote the collection of all orientation preserving covering maps $f: \T \to \T$ which satisfy:
\begin{enumerate}
\item there exist a finite set $\PP$ and an integer $N\ge 0$ so that $\PP_0:=f^N(\PP)$ consists of fixed points of $f$ and
so that the map $f$ is $C^{1+\alpha}$ on each component of $\T\setminus \PP$;
 \item for each $p_0\in \PP_0$,
 $$Df(p_0^-)\ge Df(p_0^+)\ge 1,$$
 and for each $x\in \T\setminus \PP_0$, $$Df(x^-)\ge Df(x^+)>1;$$
\item\label{classEalpha} for each $p_0\in \PP_0$, if $Df(p_0^+)=1$ (resp. $Df(p_0^-)=1$), then there exists $A_+>0$ (resp. $A^->0$) so that
$$
\lim_{x\downarrow p_0}\dfrac{Df(x)-1}{ d(x,p_0)^{\alpha}}\to A_+\,\,  \left(\mbox{resp. }\lim_{x\uparrow p_0}\dfrac{Df(x)-1}{ d(x,p_0)^{\alpha}}\to A_-\right).$$

\end{enumerate}
Here we assume that $Df(x_0^+)$ (resp. $Df(x_0^-)$) exists and is defined as $\lim_{x\downarrow x_0} Df(x)$ (resp. $\lim_{x\uparrow x_0} Df(x)$).

These maps serve as the simplest examples of non-uniformly expanding dynamical systems
with {\lq}intermittency{\rq}
and are often used to test the efficiency of a method. Examples are the famous Pomeau-Manneville map \cite{PM}
$$f(x)=x+x^{1+\alpha}\mod 1$$ and the following map
popularised by Liverani-Saussol-Vaienti \cite{LSV}:
$$f(x)= \left\{ \begin{array}{lc}
x(1+2^\alpha x^{\alpha})&\mbox{ for }x\in [0,1/2),\\
2x-1 &\mbox{ for }x\in [1/2,1].
\end{array} \right.
$$


Under some additional conditions it is well known that
a map $f\in \sE_{\alpha}$
has a {\em unique physical measure} $\mu_f$: for Lebesgue a.e. $x\in \T$, we have
$$\frac{1}{n}\sum_{i=0}^{n-1}\delta_{f^i(x)}\to \mu_f\text{ as } n\to\infty$$
in the weak star topology.
When $\alpha\ge 1$ at one of the fixed points, then $\mu_f$ is the Dirac measure at $p_0$ while for
$\alpha\in (0,1)$, $\mu_f$ is absolutely continuous with respect to the Lebesgue measure, see  \cite{Pi} and also \cite{Zw}.
In fact, it was even shown that for $\alpha\in (0,1)$
these maps are  mixing with polynomial decay of correlation,
see for example~\cite[Theorem 6]{Young2},\cite{Young3}, \cite{Hu}  and \cite{LSV}. For large deviation results concerning
such maps,   see
\cite{MN}, \cite{Mel}, \cite{PS} and  \cite{DGM}.


In this paper we will also consider this situation under {\em stochastic perturbations}.
More precisely,
let $f_t(x)=f(x)+t\mod 1$. For each $\eps>0$ let $\theta_\eps$ denote the normalized Lebesgue measure on $[-\eps,\eps]$.
Let $\Omega_\eps^\N=[-\eps,\eps]^\N$ and let $\theta_\eps^\N$ be the product measure
on this space. 
Let $F: \T\times \Omega_\eps^\N\to \T\times \Omega_\eps^\N$ denote the map
$$(x, t_0, t_1, \ldots)\mapsto (f_{t_0}(x), t_1, t_2, \ldots).$$
Write
$$f_\bft^n=f_{t_{n-1}}\circ \cdots f_{t_1}\circ f_{t_0}.$$
We denote by $\mathbb{P}_\eps$
the measure $\Leb \times \theta^\N_\eps$ on $\T\times \Omega_\epsilon^\N$, where
$\Leb$ is the Haar measure on $\T$.

A Borel probability measure $\mu_\eps$ on $\T$ is called a {\em stationary measure} for $\theta_\eps$, if for each Borel subset $E$ of $\T$, we have
$$\mu_\eps(E)=\int_{-\eps}^\eps \mu_\eps (f_t^{-1}(E)) d\theta_\eps(t).$$
If $f$ has a unique physical measure $\mu_f$, then
we say that $f$ is {\em stochastically stable} with respect to $(\theta_\eps)_{\eps>0}$ if for each $\eps>0$ small enough, there exists a unique stationary measure $\mu_\eps$ for $\theta_\eps$ and $\mu_\eps\to \mu_f$ as $\eps\to 0$ in the weak star topology.
We say that  $f$ is {\em strongly stochastically stable}
if $\mu_\eps \to \mu_f$ in the {\em strong topology}, i.e.\ if $d_{tv}(\mu_\eps,\mu)\to 0$ as $\eps\to 0$.
Here $d_{tv}(\mu_\eps,\mu)=\sup_{A} |\mu_\eps(A)-\mu(A)|$ where $A$ runs over all Borel sets.
If $\mu_\eps,\mu_f$ are absolutely continuous with densities $\zeta$
and $\zeta_\eps$, then strong convergence is equivalent to
$||\zeta-\zeta_\eps||_1\to 0$ as $\eps\to 0$ where $||\cdot ||_1$ stands for the $L^1$ norm.

The main result of this paper is the following theorem.

\begin{mainthm}
Let $f\in \sE_{\alpha}$ for some $0<\alpha<1$
and let $\theta_\eps$ be as above. Then there exists $\epsilon_0>0$
such that the following properties hold for each $\eps\in (0,\eps_0)$:
\begin{enumerate}
\item The random dynamical system $f_\bft$ has a unique stationary measure $\mu_\eps$ for $\theta_\eps$.
\item The support of the stationary measure $\mu_\eps$ is equal to $\T$, $\mu_\eps$ is absolutely continuous
and  for almost all  $(x,\bft)\in \T\times \Omega_\eps^\N$,
$$\dfrac{1}{n} \sum_{i=0}^{n-1} \delta_{f^i_\bft}(x) \to \mu_\eps$$
in the weak star topology.
\item $f$ has a unique absolutely continuous invariant measure.
\item $f$  is strongly stochastically stable.
\end{enumerate}
\end{mainthm}

\begin{rema}
The proof goes beyond the case of additive noise:
Let $\hat \Omega_\eps$ be a set which is bounded in the $C^{1+\alpha}$ norm and which is contained in
an  $\eps$-neighbourhood of  $f$ in the $C^1$ topology.
Let
$\hat G\colon  \T\times \hat \Omega^\N_\eps\to \T\times \hat \Omega^\N_\eps$ be
of the form $(x,g_0,g_1,\dots)\mapsto (g_0(x),g_1,g_2,\dots)$ and consider a measure $\theta_\eps$ on $\hat \Omega_\eps$
with the property that there exists $L>0$ so that for each $x$ and each Borel set $E$ one has
 $\theta_\eps(\{g\in \hat \Omega_\eps; g(x)\in E\})\le L(|E|/\eps)^{1/L}$. The results stated in the Main Theorem go through in this setting. Indeed, the only places where the precise form for the noise  perturbations is used are
 Proposition~\ref{prop:tailE} and
Lemmas~\ref{lem:badsmall} and \ref{lem:r1large} where the same proofs go through.
\end{rema}

In \cite{ArT},  stochastic stability of Pomeau-Manneville maps $f_\alpha(x)=x+x^{1+\alpha} \mod 1$
is also discussed. For $\alpha\ge 1$, it is shown that $f=f_\alpha$ is stochastically stable.
In this case, for each $\eps>0$, the stationary measure $\mu_\eps$ is absolutely continuous but the physical
measure $\mu_f$ for $f$ is supported on the unique neutral fixed point, so the total variation
$d_{tv}(\mu_\eps,\mu_f)\to 1$. Thus $f$ is not strongly stochastically stable when $\alpha\ge 1$.
For $\alpha\in (0,1)$, it is shown that any weak star limit point of $\mu_\eps$ as $\eps\to 0$ is of the form $s\delta_0+(1-s)\mu_f$ where $\mu_f$ is the absolutely continuous invariant measure of $f$.

%
%
%
%

Over the last two decades there has been an increasing interest in
stochastic stability. The hyperbolic case is fairly well understood, see
\cite{KK}, \cite{Ki1}, \cite{Ki2},  \cite{Young1} and also \cite{Met} (for Lorenz maps) and \cite{C} (for
piecewise expanding maps in higher dimensions).
There are also some results in the non-uniformly
hyperbolic case, see  \cite{AVil}, \cite{AVi}, \cite{AA}, \cite{Alves}, \cite{AVil}.
In addition there are also quite a few result which deal with the case of quadratic interval maps and H\'enon maps,
mostly  using Benedicks-Carelson parameter elimination, see
\cite{BeY}, \cite{BaV}, \cite{BaY} and \cite{BeV}. For surveys see \cite{V} and \cite{BDV}.

Recently, strong stochastic stability was shown for a large class of non-uniformly expanding interval maps (with critical points) by the first author in~\cite{Shen}. Comparing to the maps considered there, the non-linearity and combinatorics in our current setting are much simpler. However, there is a new phenomenon to be analyzed in the current setting: a random orbit may stay in a small neighborhood of the neutral fixed point for a very long time, so during these periods, only very weak expansion can be obtained. We believe that this paper is the first to deal with stochastic stability in systems with this kind of intermittency behaviour.

\section{Strategy of proof and some preliminary definitions}

There are several approaches for proving the existence of an absolutely continuous invariant measure
in the deterministic case. The most classical of these
is to view this as a fixed point problem by considering the Perron-Frobenius operator acting
on some subspace of $L^1$ functions $\phi$:
$$\mathcal{PF}(\phi)(x)=\sum_{f(y)=x} \dfrac{\phi(x)}{|Df( y)|}.$$
The challenge then is to construct a suitable Banach space of density functions
and show that one can apply some fixed point theorem in this Banach space (whose fixed point
will be the density of the invariant measure). This approach
was used with great success by many people in the hyperbolic case, but is less flexible and appropriate in the case where one has  no or only a very weak hyperbolic structure.

Another approach is to estimate
$$\mathcal{PF}^n(\phi)(x)=\sum_{f^n(y)=x} \dfrac{\phi(x)}{|Df^n( y)|}$$
directly, by establishing upper bounds for the size of a set $f^{-n}(A)$.
This is done by decomposing the various components of $f^{-n}(A)$
and using certain first return maps.
In the deterministic case this approach was taken in \cite{NS}, \cite{BSS}, \cite{BRSS}.

A third related approach is to consider induced transformations on some {\lq}tower{\rq},
to show that the induced
transformation has an absolutely continuous invariant measure and subsequently then to
show that the inducing times are summable. This approach is also rather classical and
has gained popularity  through the work of \cite{Young2} and \cite{Young3} and was applied
successfully to interval and H\'enon  maps which are weakly hyperbolic  (namely satisfying Collet-Eckman and certain
slow recurrence conditions).

Our approach is a mixture of the previous two and follows Shen's proof in \cite{Shen}
in which he shows stochastic stability for a large class of non-hyperbolic interval maps
(with maps which satisfy assumptions which are much weaker than the usual Collet-Eckmann conditions).

One issue our proof needs to overcome is that a map $f\in \sE_\alpha$ is not differentiable everywhere,
which means that when $n$ is large,  $f^n_\bft$ can have many points at which it is not differentiable.


\subsection{Distortion}

In the following we fix $0<\alpha<1$ and $f\in \sE_{\alpha}$.
As usual, we shall identify an interval in $\R$ of length  $\le 1$ with its projection on $\T$.
For a $C^{1+\alpha}$ diffeomorphism  $h: J\to J'$, define the {\em distortion} as
$$\Dist(h|J)=\sup_{x,y\in J} \log (|h'(x)|/|h'(y)|).$$
If $J\subset \T$ is an interval which is disjoint from
$$\PP_*=\{p\in \PP: Df(p^-)>Df(p^+)\},$$
and such that $f|J$ is injective,
then $f:J\to f(J)$ is a $C^1$ diffeomorphism which is $C^{1+\alpha}$
outside $P_0$ and therefore $\Dist(f|J)\le C  |J|^{\alpha}$, where $C$ is a constant depending only on $f$.
It follows that if $f^n|J$ is injective and none of the intervals $J,\dots,f^{n-1}(J)$
intersect $\PP_*$ then $f^n|J$ is a $C^{1}$ diffeomorphism onto its image and
$$\Dist(f^n|J) \le C\sum_{i=0}^{n-1} |f^i(J)|^\alpha.$$
Of course, if $J$ intersects $\PP_*$ then $f|J$ fails to be $C^{1+\alpha}$, but due to our assumption on the local
properties of $f$ near $p\in\PP$, the sum $\sum_{i=0}^{n-1} |f^i(J)|^\alpha$ still provides distortion control of $f^n|J$,
see Section \ref{subsec:compactinL1}.

For any $x\in \T,$ $\bft\in \Omega_\eps^\N$ and $\tau\in (0,1/2)$, we shall use
$U_\tau^{(n)}(x, \bft)$ to denote the component of $(f_\bft^n)^{-1}(B_\tau(f_\bft^n(x)))$ containing $x$. Note that $f_\bft^n$ maps $U_\tau^{(n)}(x,\bft)$ homeomorphically onto $B_\tau(f_\bft^n(x))$.
Let
$$\Lambda_\tau^{(n)}(x,\bft)=\inf\{Df_\bft^n(y^+): y\in U_\tau^{(n)}(x,\bft)\},$$
and let
\begin{equation}\label{eqn:defsL}
\sL_\tau^{(n)}(x,\bft)=\sup_{y\in U_\tau^{(n)}(x,\bft)}\sum_{j=0}^{n-1} \left(\frac{1}{Df_{\sigma^j(\bft)}^{n-j} (f^j_\bft(y^+))}\right)^\alpha.
\end{equation}
$\Lambda_\tau^{(n)}(x,\bft)$ is a lower bound for the expansion and $\sL_\tau^{(n)}(x,\bft)$ will be used to control distortion.

We shall use the following observation repeatedly.
\begin{lemma}\label{lem:slrec}
For any $x\in \T$, $\bft\in\Omega^\N$, positive integers $m$, $n$ and $\tau>0$, we have
$$\sL_{\tau}^{(m+n)}(x,\bft)\le \sL_{\tau}^{(m)}(f_\bft^n(x),\sigma^n\bft) + \Lambda^{-\alpha}\sL_{\tau'}^{(n)}(x,\bft),$$
where $\Lambda=\Lambda_\tau^{(m)} (f^n_\bft(x),\sigma^n(\bft))$ and $\tau'=\Lambda^{-1}\tau$.
In particular,
$$\sL_{\tau}^{(m+n)}(x,\bft)\le \sL_{\tau}^{(m)}(f_\bft^n(x),\sigma^n\bft) + \Lambda^{-\alpha}\sL_{\tau}^{(n)}(x,\bft).$$
\end{lemma}
%

\subsection{Escape times}\label{subsec:escape}
Take $\hI$ to be a neighbourhood of $\PP_0$
so that each component of $\hI$ contains a unique fixed point
and so that  (i) $f(\hI)$ compactly contains $\hI$ (ii) $f|\hI$ is injective
(iii) for any $x\in I$, $f(x)\notin I$ implies $f^2(x)\not\in I$
(iv) if $x,f(x)\in I$, then they are  contained in the same components of $I$.
By shrinking $I$ if necessary, we assume that these properties
(i)-(iv) are still satisfied when $f$ is replaced by $f_t$ with $|t|$ small.

\begin{definition}
The {\em first escaping time of } $(x,\bft)\in \T\times [-\eps_0,\eps_0]$ is defined as
$$\hE(x,\bft):=\inf\{m\ge 0: f_\bft^m(x)\not \in \hI\}$$
and the {\em first essential return} as $r_1(x,\bft)=\inf\{s\ge E(x,\bft): f_\bft^s(x)\in I\}$.
\end{definition}

\begin{definition}
The {\em $k$-th essential return time} $r_{k,\bft}=r_{k,\bft}(x)$  for $x\in \T$ and $\bft\in \Omega_\eps^\N$ is defined
inductively as follows. For each $k\ge 2$,  define
$$F_{k-1}(x,\bft)=f_\bft^{r_{k-1}(x,\bft)}(x),\mbox{ and }$$
$$r_k(x,\bft)=\inf\{s\ge r_{k-1}(x,\bft) + E(F_{k-1}(x,\bft)): f_\bft^s(x)\in I\}.$$
\end{definition}

To reduce the notation in the following argument, we shall introduce two constants $\tau_*$ and $\lambda_*$ now.
Let $\tau_*>0$ be a constant smaller than the distance from $f^{-1}(I)\cap I$ to $\T\setminus I$  and let $\lambda_*:=\inf\{Df(x): x\not\in f^{-3}(I)\cap I\}>1.$
\begin{lemma}\label{lem:constantstau*lambda*}
Provided that $\eps>0$ is small enough, for each $x\in I$ and $\bft\in \Omega_\eps^\N$ with $E(x,\bft)<\infty$, we have $$\Lambda_{\tau_*}^{(E(x,\bft))}\ge \lambda_*.$$
\end{lemma}
\begin{proof}
Write $E=E(x,\bft)$. Provided that $\eps>0$ is small enough, $f_\bft^{E-1}(U_{\tau_*}^{(E)}(x,\bft))$ is disjoint from $f^{-3}(I)$. Thus for each $y\in U_{\tau_*}^{(E)}(x,\bft)$,
$Df_\bft^E(y)\ge Df_{t_{E-1}}(f_\bft^{E-1}(y))\ge \lambda_*.$
\end{proof}
Write $\sL^{(n)}(x,\bft)=\sL_{\tau_*}^{(n)}(x,\bft)$, $\Lambda^{(n)}(x,\bft)=\Lambda_{\tau_*}^{(n)}(x,\bft)$ and $U^{(n)}(x,\bft)=U^{(n)}_{\tau_*}(x,\bft)$.

\begin{definition}\label{def:mtauK}
An integer $m\ge 1$ is  called a {\em $K$-scale expansion time } of  $(x,\bft)$ if
$$\sL^{(m)}(x,\bft)\le K,
\mbox{ and }
m\ge E(x,\bft).$$
We define
$$m_{K}(x,\bft):=\inf\{m; m \text{ is a $K$-scale expansion time of } (x,\bft)\}.$$
\end{definition}

\subsection{The Key Estimate}
The main step in the proof of the Main Theorem is  the following tail estimate:

\begin{theorem}\label{thm:tail-estimate} For any $\tilde \alpha \in (1,\alpha)$, there exist $K_0>0$ and $C_0>0$  
such that for each $m\ge 1$, we have $$\mathbb{P}_\eps \left((x,\bft) ;\, m_{K_0}(x,\bft)\ge m\right)\le C_0m^{-1/\tilde \alpha},$$
provided that $\eps>0$ is small enough.
\end{theorem}

\subsection{Organisation of this paper}
In Section~\ref{sec:escape} we prove first expansion and distortion estimates
associated to the first escaping from the parabolic fixed point. Then in Section~\ref{sec:essentialret}, we show that
for a certain large set of  parameter set the distortion associated to certain essential returns has control. These estimates
will then be used in  Section~\ref{proof:thm-tail-estimate} to obtain the required
tail estimate from Theorem~\ref{thm:tail-estimate}.
In Theorem~\ref{thm:tail-estimate-nice} we construct an induced map with Markov properties and which still
satisfies a similar tail estimate.
The main theorem will then be deduced from Theorem~\ref{thm:tail-estimate-nice} in
the final part of Section~\ref{sec:final}.

\subsection{Notation}
	
Given functions $a,b\colon \R \to \R$ (or sequences $a_n,b_n$),
we write $a\asymp b$ (resp. $a\lesssim b$) if there exists a universal constant  $C$ so that
$1/C\le |a/b|\le C$ (resp. $|a|\le C|b|$).
We denote by $\sigma\colon \Ome\to \Ome$ the shift map $\sigma(t_0,t_1,\dots)=(t_1,t_2,\dots)$.
When ${\mathbf V}\subset T\times \Ome$, define ${\mathbf V}^\bft={\mathbf V}\cap (T\times \{\bft\})$.
Finally, given a set $I$ we will denote by $I(x)$ the component of $I$ that contains $x$.

\section{Estimates associated to escaping times}
\label{sec:escape}

Throughout this section, we fix a constant $\kappa\in (\alpha, 1)$ such that
\begin{equation}\label{eqn:kappa}
\kappa(1+\alpha)>1
\end{equation}
so that $\eps^{-\kappa \alpha}>>\eps^{-\alpha/(1+\alpha)}$ for $\eps>0$ small.

\subsection{One passage away from a parabolic fixed point}

In this section we analyse the distortion and expansion while an orbit escapes a component of $I$
(where $I$ is the neighbourhood of the set of fixed points in $\PP_0$ defined in Subsection~\ref{subsec:escape}).

\begin{prop} \label{prop:escape}
There exist constants $\Delta_*>0$ and $K_*>1$ such that the following hold for any $x\in I$ and $\bft\in \Omega_\eps^\N$,  provided that $\eps>0$ is small enough.
\begin{enumerate}
\item If $\dist(f(x),x)\ge \Delta_*\eps$, then
$$\hE(x,\bft)\le \eps^{-\alpha/(1+\alpha)}\mbox{ and }\Lambda^{(E(x,\bft))}(x,\bft)\ge K_*^{-1} \dist(x, \PP_0)^{-1}.$$
\item If $\hE(x,\bft)\le \eps^{-\kappa\alpha}$, then $m_{K_*}(x,\bft)= E(x,\bft).$
\end{enumerate}
\end{prop}

Let $I^{\parab}$ denote the union of all components $U$ of $I\setminus \PP_0$ for which $\inf_{y\in U} Df(y)=1$ and let  $I^{\rep}$ denote the union of all other components of $I\setminus \PP_0$.
Moreover, let
$$I_\delta=\{x\in I: \dist(f(x), x)\le \delta\}.$$
So $I_\delta\supset \PP_0$.

\begin{lemma} \label{lem:bindescape}
There exists $\Delta>0$ and $C>0$ such that the following holds for any $\bft\in\Omega_\eps^\N$, provided that $\eps>0$ is small enough.
If $x\in  I^\parab\setminus I_{\Delta\eps}$, then
\begin{enumerate}
\item $C^{-1}\dist(x,\PP_0)^{-\alpha}\le E(x,\bft)\le C \dist(x, \PP_0)^{-\alpha},$
\item $Df_\bft^{E(x,\bft)}(x)\ge C^{-1} \dist(x,\PP_0)^{-\kappa (1+\alpha)};$
\end{enumerate}
and if $x\in I^\rep\setminus I_{\Delta\eps}$, then
\begin{enumerate}
\item[(1')] $C\log \dist(x,\PP_0)^{-1} \le E(x,\bft)\le C^{-1}\log \dist(x,\PP_0)^{-1},$
\item[(2')] $Df_\bft^{E(x,\bft)}(x)\ge C \dist(x,\PP_0)^{-1}$.
\end{enumerate}
\end{lemma}

\begin{proof}
We shall only deal with the case $x\in I^\parab$ as the other case is simpler. Let $p_0\in \PP_0$ be such that $x\in I(p_0)$. For simplicity of notation, we shall assume $p_0=0$ and $x>0$. Let $A:=\lim_{y\to 0^+} (Df(y)-1)/y^\alpha$.
Fix a small constant $\eta>0$ such that
$\kappa\le (A-\eta)/(A+2\eta)$.
Then there exist $b>0$ small and $\Delta>0$ large such that when $0<x\le 2b$, $f(x)-x\ge \Delta\eps$ and $|t|\le \eps$, we have
$$x+2\eps\le x+ A_0 x^{1+\alpha}\le f_t(x)\le x+ {A_1}x^{1+\alpha},$$
where $A_0=\frac{A-\eta}{1+\alpha}$ and $A_1=\frac{A+\eta}{1+\alpha}$. For $x\in (-b, b)$ and $\bft\in\Omega_\eps^\N$, let $m_0(x,\bft)$ be the minimal positive integer such that $f_\bft^{m_0}(x)>b$.
Note that $E(x,\bft)-m_0(x,\bft)$ is a non-negative integer which is bounded from above (when $b$ is fixed).
So we may assume $x\le b$ and show that the assertions are true when $E(x,\bft)$ is replaced by $m_0(x,\bft)$.

To this end, let us fix $(x,\bft)$ and write $m_0=m_0(x,\bft)$, $x_i=f_\bft^i(x)$. Then
$$x_i+A_0 x_i^{1+\alpha} \le x_{i+1}\le x_i+ A_1 x_i^{1+\alpha},$$
hence
$$\frac{1}{x_i^\alpha} \frac{1}{(1+A_1 x_i^\alpha)^\alpha}\le \frac{1}{x_{i+1}^\alpha}\le \frac{1}{x_i^\alpha} \frac{1}{(1+A_0 x_i^\alpha)^\alpha}.$$
Reducing $b>0$ if necessary, this implies that
$$(A-2\eta) \frac{\alpha}{1+\alpha}\le \frac{1}{x_{i}^\alpha}-\frac{1}{x_{i+1}^\alpha}\le (A+2\eta) \frac{\alpha}{1+\alpha}.$$
Summing over $i=0,1,\ldots, m_0-1$, we obtain
$$(A-2\eta)\frac{\alpha}{1+\alpha} m_0 \le \frac{1}{x^\alpha} -\frac{1}{x_{m_0}^\alpha} \le (A+2\eta)\frac{\alpha}{1+\alpha} m_0.$$
Thus
$$\frac{1+\alpha}{(A+2\eta)\alpha m_0 + (1+\alpha) b^{-\alpha}}\le x^\alpha \le \frac{1+\alpha}{(A-2\eta)\alpha m_0},$$
which implies (1).
Similarly,  for $j=0,\dots,m_0-1$,
$$\frac{1+\alpha}{(A+2\eta)\alpha (m_0-j) + (1+\alpha) b^{-\alpha}} \le x_j^\alpha\le \frac{1+\alpha}{(A-2\eta)\alpha (m_0-j)}.$$
Thus
\begin{multline*}
|Df_\bft^{m_0}(x)| =\prod_{j=0}^{m_0-1}|Df_{t_j}(x_j)|\ge \prod_{j=0}^{m_0-1} \left( 1+ (A-\eta) x_j^{\alpha}\right)
\\ \gtrsim  m_0^{\kappa(1+\alpha)/\alpha}\gtrsim |x|^{-\kappa(1+\alpha)}.
\end{multline*}
This proves the statement (2).
\end{proof}

\begin{proof}[Proof of Proposition~\ref{prop:escape}]
The statement (1) follows directly from Lemma~\ref{lem:bindescape} by choosing $K_*$ suitably large.
To prove the statement (2), we shall only consider the case $x\in I^\parab$, as the other case is simpler.
We need the following two Claims.

{\bf Claim 1. } For each $z\in I$ and $\bfu\in \Omega_\eps^\N$ and $z'\in U^{(E)}(z,\bfu)$, we have
\begin{equation}\label{eqn:m0compare}
E(z,\bfu)+1\ge E(z',\bfu)\ge E(z,\bfu)-1
\end{equation}
for $\eps>0$ small.
To prove the claim, write $E:=E(z,\bfu)$ and $E': = E(z',\bft)$. Without loss of generality, assume $E'< E$. Note that $\dist(f_\bfu^{E'}(z'), f_\bfu^{E'}(z))\le \dist(f_\bfu^{E}(z'), f_\bfu^E(z))< \tau_*$. Since $f_\bft^{E'}(z')\not\in I$, the choice of $\tau_*$ guarantees that
$E-E'=E(f_\bfu^{E'}(z'),\sigma^{E'}\bfu)=1$, provided that $\eps>0$ is small enough.
This proves Claim 1.

{\bf Claim 2.} There exists a constant $c>0$ such that for each $z\in I^\parab$, $\bfu\in\Omega_\eps^\N$, we have
$$Df_\bft^{E(z,\bfu)}(z)\ge c \min (E(z,\bfu)^{\kappa(1+\alpha)/\alpha},\eps^{-\kappa}).$$
provided that $\eps>0$ is small enough.

Indeed, by Lemma~\ref{lem:bindescape}, there exists $\eta>0$ such that the estimate holds when $E(z,\bfu)\le \eta \eps^{-\alpha/(1+\alpha)}$.
When $E(z,\bfu)> \eta \eps^{-\alpha/(1+\alpha)}$, there exists a minimal $k$ such that $E(f_\bfu^k(u),\sigma^k\bfu)=E(z,\bfu)-k\le \eta \eps^{-\alpha/(1+\alpha)}$, so that
$$Df_\bfu^{E(z,\bfu)}(y)\ge Df_{\sigma^k\bfu}^{E(z,\bfu)-k} (f_\sigma^k(z))\gtrsim (E(z,\bfu)-k)^{\kappa(1+\alpha)/\alpha}\asymp \eps^{-\kappa}.$$
Thus the claim holds.

Now let us fix $x\in I^\parab$ and $\bft\in \Omega_\eps^\N$ with $E:=E(x,\bft)\le \eps^{-\kappa \alpha}$.
Then for each $0\le i< E$ we have
$$E(f_\bft^i(x),\sigma^i\bft)=E-i.$$
For each $y\in U:=U^{(E)}(x,\bft)$, $f_\bft^i(y)\in U^{(E-i)}(f_\bft^i(x),\sigma^i(\bft))$,  and so by Claim 1,
$$E-i-1\le E(f_\bft^i(y),\sigma^i\bft)\le E-i+1.$$
Applying Claim 2 to $(z,\bfu)=(f_\bft^i(y),\sigma^i(y))$, we obtain that
$$Df_{\sigma^i\bft}^{(E-i)}(f_\bft^i(y))\asymp Df_{\sigma^i\bft}^{(E(z,\bfu))}(f_\bft^i(y))\ge c'' \min ((E-i)^{\kappa(1+\alpha)/\alpha},\eps^{-\kappa}).$$
Thus, because $E\le \eps^{-\kappa \alpha}$,
$$\sL^{(E)}(x,\bft)\lesssim \sum_{i=0}^{E-1} \max\left((E-i)^{-\kappa(1+\alpha)}, \eps^{\kappa\alpha} \right)\asymp 1.$$
\end{proof}

\subsection{Tail estimate for escaping time}
\begin{prop}\label{prop:tailE}
There exists a constant $C>0$ such that
 $$\mathbb{P}_\eps \left(\{E(x,\bft)\ge m\}\right) \le \max\left( C m^{-1/\alpha}, 8\eps\right).$$
\end{prop}
\begin{proof}
Fix $p_0\in \PP_0$. Let $X_m=\{(x,\bft)\in I(p_0)\times \Omega_\eps^\N: E(x,\bft)\ge m\}$.
For simplicity of notation, assume $p_0=0$. Let $g: I(0)\to f^{-1}(I(0))\cap I(0)$ denote the inverse branch of $f$.
By our assumption on $f$,  there exists
$A_0>0$ such that
$$Dg(x)\le 1- A_0 |x|^{\alpha}\mbox{ for all } x\in I(0).$$

For each interval $J\subset I(0)$, let $\hJ=J\setminus [-2\eps, 2\eps]$.

{\bf Claim.} There exists a constant $A_1>0$ such that for each interval $J\subset I(0)$, we have 
\begin{equation}\label{eqn:inverseJ}
\int_{-\eps}^\eps |g(J+t)| d\theta_\eps(t)\le |J|-A_1|\hJ|^{1+\alpha}.
\end{equation}
Indeed, for each $y\in\hJ$ and each $t\in [-\eps, \eps]$, we have
$|y+t|\ge |y|/2$, hence
$$\int_{-\eps}^\eps |y+t|^\alpha d\theta_\eps(t)\ge \left(\frac{|y|}{2}\right)^\alpha.$$
Thus
\begin{align*}
\int_{-\eps}^\eps |g(J+t)| d\theta_\eps(t)& =
\int_{y\in J} \int_{-\eps}^\eps g'(y+t) d\theta_\eps(t) dy\\
&=\int_{J\setminus \hJ} \int_{-\eps}^\eps g'(y+t) d\theta_\eps(t) dy
+\int_{\hJ} \int_{-\eps}^\eps g'(y+t) d\theta_\eps(t) dy\\
& \le |J\setminus \hJ| + \int_{\hJ} \int_{-\eps}^\eps (1- A_0 |y+t|^\alpha) d\theta_\eps(t) dy\\
& \le |J|-\frac{A_0}{2^\alpha} \int_{y\in \hJ} |y|^\alpha dy\\
& \le |J|-A_1|\hJ|^{1+\alpha},
\end{align*}
where $A_1=A_0 2^{-\alpha} /(1+\alpha)$. The claim is proved.

It follows that
\begin{align*}
\mathbb{P}_\eps(X_{m+1})& = \int_{\bfs\in \Omega_\eps^\N} \int_{-\eps}^\eps |g(X_m^\bfs +t)| d\theta_\eps(t) d\theta_\eps^\N(\bfs)\\
&\le \int_{\bfs\in\Omega_\eps^\N} \left( |X_m^\bfs| -A_1 |\hX_m^\bfs|^{1+\alpha}\right) d\theta_\eps^\N(\bfs)\\
& = \mathbb{P}_\eps(X_m) - A_1 \int_{\bfs\in\Omega_\eps^\N} |\hX_m^\bfs|^{1+\alpha} d\theta_\eps^\N (\bfs)\\
& \le \mathbb{P}_\eps (X_m) -A_1 \left(\int_{\bfs\in\Omega_\eps^\N} |\hX_m^\bfs| d\theta_\eps^\N(\bfs)\right)^{1+\alpha}\\
&= \mathbb{P}_\eps(X_m)-A_1 \mathbb{P}_\eps(\hX_m)^{1+\alpha},
\end{align*}
where $\hX_m=X\setminus ([-2\eps,2\eps]\times \Omega_\eps^\N)$. So if $\mathbb{P}_\eps(X_m)\ge 8\eps$ then
$\mathbb{P}_\eps(\hX_m)\ge \mathbb{P}_\eps(X_m)/2$. Therefore
\begin{equation}\label{eqn:Pxmrec}
\mathbb{P}_\eps(X_{m+1})\le \max\left( 8\eps, \mathbb{P}_\eps(X_m)-A \mathbb{P}_\eps (X_m)^{1+\alpha}\right),
\end{equation}
where $A=A_1/2^{1+\alpha}$.

Finally, choose a large constant $C>1$ such that
$$Cm^{-1/\alpha} -A(Cm^{-1/\alpha})^{1+\alpha} \le C (m+1)^{-1/\alpha}$$
holds for all $m\ge 1$.   The desired  upper bounds of $\mathbb{P}_\eps(X_m)$ follows from (\ref{eqn:Pxmrec}) by an easy induction on $m$.
\end{proof}

Combining Propositions~\ref{prop:escape} and~\ref{prop:tailE}, we obtain
\begin{coro}\label{cor:tailmsmall}
There exists $C>0$ such that the following holds provided that $\eps>0$ is small enough. For each $m\le \eps^{-\kappa\alpha}$ we have
$$\mathbb{P}_\eps \left(m_{K_*}\ge m \right)\le C m^{-1/\alpha}.$$
\end{coro}
\begin{proof} For $m\le \eps^{-\kappa\alpha}$, Proposition~\ref{prop:escape} asserts that $m_{K_*}(x,\bft)\ge m$ implies that $E(x,\bft)\ge m$.
Since $m^{-1/\alpha}\ge \eps^\kappa>>\eps$,  the inequality follows from Proposition~\ref{prop:tailE}.
\end{proof}
\section{Distortion estimates associated to essential return times}
\label{sec:essentialret}

We continue to fix a constant $\kappa\in (0,1)$ such that (\ref{eqn:kappa}) holds. We also fix an arbitrary constant $\gamma> 1-\kappa$.
\subsection{The first essential return time}
If the point $x$ lingers a long time near a neutral fixed point, then the expansion
is no longer related to $E(x,\bft)$ and the term $\sL^{(m)}(x,\bft)$ controlling
distortion also gets large in terms of $m$:

\begin{lemma}\label{lem:distgeneral}
There exists a constant $K_1>0$ such that the following holds provided that $\eps>0$ is small enough.
For $x\in \T$ and $\bft\in [-\eps,\eps]^\N$, if $m$ is a positive integer such that $E(x,\bft)\le m\le r_1(x,\bft)$, then
$$\sL^{(m)}(x,\bft)\le K_1+ K_1m\eps^{\kappa \alpha}.$$
Moreover, if $x\in I\setminus I_{\Delta_*\eps}$, then $\sL^{(m)}(x,\bft)\le K_1$, and if $x\in I_{\Delta_*\eps}$ then
$\Lambda^{(m)}(x,\bft)\ge K_1^{-1}\eps^{-\kappa}.$
\end{lemma}
\begin{proof}
Let $E=E(x,\bft)$. Suppose $m>E$. Then for each $y\in U^{(m-E)}(f_\bft^E(x), \sigma^E\bft)$, and $0\le k< m-E$, we have $f_{\sigma^E\bft}^k(y)\not\in f^{-1}(I)\cap I$, since
$f_{\sigma^E\bft}^k(f_\bft^E(x))\not\in I$ and $\dist(f_{\sigma^E\bft}^k(y),f_{\sigma^E\bft}^k(f_\bft^E(x)) )\le \tau_*$. Thus
$Df_{\sigma^E\bft}^k(y)\ge \lambda_*^k$ for all $0\le k<m-E$, which implies that
$$\sL^{(m-E)}(f_\bft^E(x), \sigma^E\bft)\le \sum_{j=1}^{m-E}\lambda_*^{-j\alpha}<\lambda_*^\alpha/(\lambda_*^\alpha-1),$$
and $$\Lambda^{(m)}(x,\bft)\ge \Lambda^{(m-E)}(f_\bft^E(x), \sigma^E\bft)\ge \lambda_*^{m-E}.$$
Since
$$\sL^{m}(x,\bft)\le \lambda_*^{-(m-E)\alpha} \sL^{E}(x,\bft) + \sL^{(m-E)} (f_\bft^E(x),\sigma^E\bft),$$
it suffices to prove the lemma in the case that $x\in I$ and $m=E$.

If $x\not\in I_{\Delta_*\eps}$, then by Proposition~\ref{prop:escape},
$\sL^{(E)}(x,\bft)\le K_*.$ So the lemma holds in this case. 
Now assume that $x\in I_{\Delta_*\eps}$. Let $k\ge 1$ be the minimal positive integer such that
$y:=f_\bft^k(x)\not\in I_{\Delta_*\eps}$ and let $\bfs=\sigma^k\bft$. Then $E:=E(x,\bft)=E(y,\bfs)+k$ and $\dist(f(y),y)\asymp \eps$.
By Proposition~\ref{prop:escape},
$\sL^{(E-k)}(y,\bfs)\le K_*$ and $\Lambda:=\Lambda^{(E-k)}(y,\bfs)\gtrsim \eps^{-\kappa}$.
Thus $$\Lambda^{(E)} (x,\bft)\ge \Lambda^{(E-k)}(y,\bfs)\gtrsim \eps^{-\kappa}.$$
Since $\sL^{(k)} (x,\bft)\le k\le E,$ we have
$$\sL^{(E)} (x,\bft)\le \sL^{(E-k)}(y,\bfs) + \Lambda^{-\alpha} \sL^{(k)} (x,\bft)\lesssim 1+ E\eps^{\kappa\alpha}.$$
Thus the lemma holds.
\end{proof}

\subsection{Distortion control outside a BAD set in $\Omega_\eps^\N$}\label{subsec:bad}

Let
$$\Omea(n)=\left\{\bft\in \Omega_\eps^\N: \sum_{k=0}^{n-1} |t_k|\le \frac{\eps n}{4}\right\}.$$
Clearly, for any $\bft\notin \Omea(n)$  one has
$$\#\{i; 0\le k\le n \mbox{ and }|t_k|\ge \epsilon/8\}\ge n/8.$$
Let $\varphi(\eps)=\log (1/\eps)$ and consider the following BAD set of $\textbf{t}$:
\begin{equation*}
\begin{array}{ll}
\BAD_\eps(N)&=\left\{ \bft: \exists 0\le i<  N \eps^{-\alpha} \varphi(\eps)\text{ and }\exists n\ge N\eps^{-\alpha} \right. \\
 &\quad \quad  \left. \mbox{ so that }  \sigma^i(\bft)\in \Omea(n)\right\}. \end{array}
\end{equation*}
\begin{lemma}\label{lem:badsmall} There exists $\rho>0$ such that for each $N=1,2,\ldots$,
$$\theta_\eps^\N \left(\BAD_\eps (N)\right)\le e^{-\rho N \eps^{-\alpha}}.$$
\end{lemma}
\begin{proof} When the random variable $t$ is uniformly distributed over $[-\eps,\eps]$
then the expected value $\int_{t\in [-\eps,\eps]} |t|\, d\theta_\eps(t)$
of $|t|$ is equal to  $\eps/2$. Hence,
by the Large Deviation Principle, there exists $\rho_1>0$ such that $\theta_\eps^\N(\Omea(n))\le e^{-\rho_1 n}$ for each $n\ge 1$.
Since $\theta_\eps^\N$ is invariant under the shift map $\sigma$ we have
$$\theta_\eps^\N(\BAD_\eps (N))=\sum_{0\le i< N\eps^{-\alpha} \varphi(\eps)}\,\, \sum_{n\ge N\eps^{-\alpha}} e^{-\rho_1 n}\le C N \eps^{-\alpha}\varphi(\eps) e^{-\rho_1 N\eps^{-\alpha}}.$$
\end{proof}
For each integer $m$, let
\begin{equation}
\tm=\left[\frac{m\eps^\alpha}{\varphi(\eps)}\right]+1
\label{eq:tilde}
\end{equation}
where $\left[t \right]$
stands for the largest integer $\le t$ and so $t< [t]+1 \le t+1$.  Note that $\tm(x,\bft) \le
[ m \eps^\alpha/\varphi(\eps)]+1
 <<m$ when $\eps>0$ is small.

\begin{lemma} \label{lem:expnotbad}
There exists a constant $\xi>0$ such that if $\bft\in\Omega_\eps^\N\setminus \BAD_\eps(N)$ for some positive integer $N$, then for any $x\in \T$,
any $0\le i< N\eps^{-\alpha} \varphi(\eps)$ and any $n\ge N \eps^{-\alpha}$, we have
$$Df_{\sigma^i(\bft)}^n(f_\bft^i(x))\ge \exp (\xi n\eps^\alpha).$$
\end{lemma}
\begin{proof}
The assumption $\bft\notin \BAD_\epsilon(N)$ implies
$\sigma^i(\bft)\notin \Omea(n)$.
So $\#\{0\le j< n; |t_{i+j}|\ge \eps/8\}\ge  n/8$. For each $j$ in this set,
either
$\dist(x_j,\PP_0)\ge \eps/20$ and $\dist(x_{j+1},\PP_0)\ge \eps/20$ and so the
derivative of $Df_\bft$ in $x_j$ or in $x_{j+1}$ is $\ge 1+C_1\eps^\alpha\ge \exp(C_2\eps^\alpha)$.
This happens at least $n/16$ times and for the other times
the derivative of $Df_\bft$ is $\ge 1$. Thus $Df^n_{\sigma^i(\bft)}(f_\bft^i(x))\ge \exp(\xi n\eps^\alpha)$ holds for some suitably chosen constant $\xi>0$.
\end{proof}

\begin{prop}
\label{prop:expansionnotbad}
There exists a constant $K_2>1$ such that for any $x\in \T$ and $\bft\in \Omega_\eps^\N$,
if $m=r_k(x,\bft)$ for some $k\ge 1$ and $\bft\not\in \BAD_\eps(\tm)$, then $$\sL^{(m)}(x, \bft)\le K_2\tm\eps^{-(1-\kappa)\alpha}.$$
\end{prop}
\begin{proof}
Let $r_0=0$ and for each $i=1,2,\ldots, k$, write $r_i=r_i(x,\bft)$,  $\sL_i=\sL^{(r_i-r_{i-1})}(f_\bft^{r_{i-1}}(x),\sigma^{r_{i-1}}\bft)$,
and $\Lambda_i=\Lambda^{(r_i-r_{i-1})}(f_\bft^{r_{i-1}}(x),\sigma^{r_{i-1}}\bft). $ Then $\Lambda_i\ge \lambda_*>1$ for all $1\le i\le k$.

{\em Case 1.} Assume $\Lambda_i < K_1^{-1}\eps^{-\kappa}$ for all $1\le i\le k$.
Then by Lemma~\ref{lem:distgeneral}, $\sL_i\le K_1$ holds for all $i=1,2,\ldots, k$.
Thus
$$\sL^{(m)}(x,\bft)\le \sL_k+\sum_{i=1}^{k-1} \sL_i (\prod_{i<j\le k}\Lambda_j)^{-\alpha}
\le C\sum_{i=1}^k \lambda_*^{-(i-1)\alpha} < K_1\lambda_*^\alpha/(\lambda_*^\alpha-1)$$
is bounded from above. Note that we do not need $\bft\not\in \BAD(\tm)$ in this case.

{\em Case 2.} There exists $i\in \{1,2,\ldots, k\}$ such that $\Lambda_i\ge K_1^{-1}\eps^{-\kappa}$. Let $i$ be maximal with this property. Let
$m_0$ be the maximal integer such that $r_{i-1}\le m_0< r_i$ and $\Lambda^{(r_i-m_0)}(f_\bft^{m_0}(x), \sigma^{m_0}\bft)\ge K_1^{-1}\eps^{-\kappa}$. Then as in Case 1, $\sL^{(m-m_0-1)}(f_\bft^{m_0+1}(x), \sigma^{m_0+1}\bft)$ and therefore
$\sL^{(m-m_0)}(f_\bft^{m_0}(x),\sigma^{m_0}\bft)$ is bounded from above by a constant.

Now let us use Lemma~\ref{lem:expnotbad} to estimate $\sL^{(m_0)}(x,\bft)$. Indeed, that lemma gives that
$D_n:=\Lambda^{(m_0-n)}(f_\bft^n(x),\sigma^n \bft)\ge \exp (\xi n\eps^\alpha)$ when $m_0-n\ge \tm \eps^{-\alpha}$.
When $m_0-n< \tm \eps^{-\alpha}$ we nevertheless have $D_n\ge 1$. Thus
$$\sL^{(m_0)}(x,t)\le \sum_{n=0}^{m_0-1} D_n^{-\alpha}\le \tm \eps^{-\alpha} + \sum_{n=1}^\infty e^{-\alpha \xi n \eps^\alpha} \lesssim \tm \eps^{-\alpha}.$$
Therefore,
$$\sL^{(m)}(x,\bft)\le \sL^{(m-m_0)}(f_\bft^{m_0}(x),\sigma^{m_0}\bft)+ K_1^\alpha \eps^{\kappa \alpha} \sL^{(m_0)}(x,t)\lesssim \tm \eps^{-(1-\kappa)\alpha}.$$
\end{proof}

\subsection{Repeated passages away from neutral points}
For $x\in \T$ and $\bft\in \Ome$, let as before
$$F_0(x,\bft)=x\mbox{ and }r_1(x,\bft)=\inf\{s\ge E(x,\bft): f_\bft^s(x)\in I\},$$
and for each $k\ge 2$, define inductively
$$F_{k-1}(x,\bft)=f_\bft^{r_{k-1}(x,\bft)}(x),\mbox{ and }$$
$$r_k(x,\bft)=\inf\{s\ge r_{k-1}(x,\bft) + E(F_{k-1}(x,\bft)): f_\bft^s(x)\in I\}.$$
The number $r_k(x,\bft)$ is called the {\em $k$-th essential return time} of $(x,\bft)$.

For each $\bft\in\Ome$, let
$$X_1^\bft(N)=\{x\in \T: \widetilde{r_1(x,\bft)}=N\}$$
where $\widetilde{r_1(x,\bft)}$ is defined as in (\ref{eq:tilde}).
\begin{lemma}\label{lem:X1t}
For each $L>0$, there exists $N_0=N_0(L)$ such that for each $N\ge N_0$ and $\bft\not\in \BAD_\eps(N)$, we have
\begin{equation}
\Leb \left(X_1^\bft(N)\right)\le \eps^L e^{-\rho N}.
\label{eq:x1t}
\end{equation}
\end{lemma}
\begin{proof} For each $m\ge 1$, let $I_m^\bft=\{x\in I: E(x,\bft)=m\}$.  For each $n\ge 1$ and $\bfs\in \Ome$, let $V_n^\bfs=\{x\not\in I: r_1(x,\bfs)=n\}.$
Since $f$ is uniformly expanding outside $I$, there exists $\rho_1>0$ such that $\Leb(V_n^\bfs)\le e^{-\rho_1 n}$ holds for all $\bfs\in\Ome$ and $n\ge 1$.
Let $X_{m,n}^\bft=\{x\in I_m^\bft: F^m(x,\bft)\in V_n^{\sigma^m\bft}\}$. Since $f_\bft^m|I_m^\bft$
is injective with derivative not less than one, we have
\begin{equation}\Leb (X_{m,n}^\bft) \le \Leb (V_n^{\sigma^m\bft})\le e^{-\rho_1 n}.
\label{eq:Xmn1}\end{equation}
Next  assume $\bft\not\in \BAD_\eps(N)$. By Lemma~\ref{lem:expnotbad}, when $m\ge N\eps^{-\alpha}$,
$|Df_\bft^m(x)|\ge \exp (\xi m \eps^{\alpha})$ holds for every $x\in \T$, thus
\begin{equation}
\Leb (I_m^\bft)\le \exp (-\xi \eps^\alpha m).\label{eq:Xmn2}\end{equation}
Since
$$X_1^\bft(N)=\bigcup_{(N-1)\eps^{-\alpha}\varphi(\eps) \le m+n< N \eps^{-\alpha} \varphi(\eps)} X_{m,n}^\bft,$$
the lemma follows.
Indeed, we can decompose this union into terms where  $n\ge m$
and those where $n<m$. 
For the first terms, we have $n> (N-1)\eps^{-\alpha}\varphi(\eps)/2$, so by (\ref{eq:Xmn1}) they give a contribution of at most $N\eps^{-\alpha}\varphi(\eps) \exp(-\rho_1 (N-1) \eps^{-\alpha} \varphi(\eps)/2)$,
while for the 2nd terms, $m\ge (N-1)\eps^{-\alpha}\varphi(\eps)/2>N\eps^{-\alpha}$, so they give a contribution of at most $\exp(-C_2 N \varphi(\eps))$ by (\ref{eq:Xmn2}).
Both these terms are bounded from above by $\epsilon^L \exp(-\rho N)$ when $\epsilon>0$ is small
and $N$ is large.
\end{proof}

\begin{lemma}\label{lem:r1large}
For any $x\in \T$ and $N\ge N_0$, we have
$$\theta_\eps^\N\left(\left\{\bft\in\Ome: r_1(x,\bft) > N \eps^{-\alpha} \varphi(\eps)\right\}\right)\le \eps e^{-\rho_2 N}.$$
\end{lemma}
\begin{proof}
Let $\Ome(x,N)= \left\{\bft\in\Ome: r_1(x,\bft) > N \eps^{-\alpha} \varphi(\eps)\right\}$.
Because of Lemma~\ref{lem:badsmall} it suffices to prove that
$\Ome(x,N)\setminus \sigma^{-1} \BAD_\eps(N)$ has exponentially small measure in $\theta_\eps^\N$.
For each $\bfs\in \Ome\setminus \BAD_\eps(N)$, consider the section
$$\Ome(x,\bfs,N):=\left\{t\in [-\eps,\eps] : t\bfs \in \Ome(x,N)\right\}.$$
For each $t\in \Ome(x,\bfs,N)$ we have $f_t(x)\in X_1^\bfs(N)$.  Therefore and by Lemma~\ref{lem:X1t}, we obtain
 $\theta_\eps^\N(\Ome(x,\bfs,N))\le (2/\eps) \Leb(X_1^\bfs(N))\le 2\eps e^{-\rho N}$, provided that $N>N_0(2)$. By Fubini's theorem, the lemma follows.
\end{proof}

\subsection{Special return times}

\begin{definition}
Given $(x,\bft)\in I\times \Ome$, the {\em special return time} $R(x,\bft)$ is the minimal essential return time $r$ of $(x,\bft)$ for which $$\bft\not\in Bad_\eps(\tr).$$
We write $$R_0(x,\bft)=R(x,\bft)\mbox{ , } G_0(x,\bft)=x$$ and for $k\ge 1$ we define inductively
$$G_k(x,\bft)=f^{R_{k-1}(x,\bft)}_\bft(x)\mbox{ and }R_k(x,\bft)=R(G_{k}(x,\bft),\bft).$$
\end{definition}

\begin{lemma} \label{lem:R-estimate}
There exist constants $R_*>0$, $C>0$ and $\rho>0$ such that
for each $x\in \T$, we have
\begin{equation}\label{eq:estimate-sumR}
\theta_\eps^\N\left( \left\{ \bft ; \sum_{i=0}^{n-1} R_i(x,\bft)> R_* n \eps^{-\alpha}\varphi(\epsilon)\right\} \right) \le C e^{-\rho n},
\end{equation}
provided that $\eps>0$ is small enough.
\end{lemma}
\begin{proof}
We first prove that there exist $\rho>0$ and $n_0>0$ such that for each $n\ge n_0$, and any $x\in \T$,
we have
\begin{equation}\label{eq:estimate-R}
\theta_\eps^\N \left(\left\{\bft\in \Ome: R(x,\bft)>n\eps^{-\alpha}\varphi(\eps) \right\}\right)\le 2 e^{-\rho n}.
\end{equation}
Let $R'(x,\bft)=0$ if $R(x,\bft)=r_1(x,\bft)$, and let $R'(x,\bft)$ be the previous essential return time of $(x,\bft)$ into $I$ otherwise.
For each $m=0,1,\ldots$, let $$E(m)=\{\bft: R'(x,\bft)=m\},\, E^*(m)=\{\bft\in E(m): r_1(f_\bft^m(x), \sigma^m \bft)> n\eps^{-\alpha} \varphi(\eps)/2\}.$$ and let
$$E_1=\bigcup_{m> n\eps^{-\alpha}\varphi(\eps)/2} E(m), E_2=\bigcup_{m\le n\eps^{-\alpha}\varphi(\eps)/2} E^*(m).$$
Then
$$\{\bft: R(x,\bft)> n\eps^{-\alpha}\varphi(\eps)\}\subset E_1\cup E_2.$$
If $\bft\in E(m)$, $m\ge 1$ then  $\bft\in \BAD_\eps(\tm)$,
since $R'(x,\bft)=m$ is not a special return time.
Hence, by Lemma~\ref{lem:badsmall},
$$\theta_\eps^\N(E(m))\le e^{-\rho \tm \eps^{-\alpha}}.$$
Therefore
\begin{equation}\label{eq:thetaepsE1}
\theta_\eps^\N(E_1)\le
\sum_{m> n\eps^{-\alpha}\varphi(\eps)/2} e^{-\rho \tm \eps^{-\alpha}}\le
\eps^{-\alpha}\varphi(\eps)
\sum_{N> n/2 } e^{-\rho N \eps^{-\alpha}}<  e^{-\rho n}/2
\end{equation}
where the $\eps^{-\alpha}\varphi(\eps)$ factor appears because for given N, the number of positive integers m with $\tm=N$
is bounded by this number.
To estimate the size of $E_2$ we recall that by Lemma~\ref{lem:r1large}, for any $y\in\T$,
$\theta_\eps^\N(\{\bfs: \widetilde{r_1(y,\bfs)}>\frac{n}{2}\})\le \eps e^{-2\rho n},$ provided that $n$ is large enough. So by Fubini's theorem,
$$\theta_\eps^\N(E^*(m))\le \eps e^{-2\rho n}$$ holds for all $m$. Thus
$$\theta_\eps^\N(E_2)\le n\eps^{-\alpha}\varphi(\eps) \eps e^{-2\rho n}< e^{-\rho n}/2.$$
The inequality (\ref{eq:estimate-R}) follows.

Now let us complete the proof of the proposition by an easy large deviation argument.
To this end, for positive integers $m_0,m_1, \ldots, m_{n-1}$, let us consider
$$E(m_0, m_1,\ldots, m_{n-1})=\{\bft: R_i(x,\bft)=m_i, 0\le i<n\}.$$
Since the conditions $R_0(x,\bft)=m_0, R_1(x,\bft)=m_1, \ldots, R_{n-2}(x,\bft)=m_{n-2}$ depends only on the first $m_0+\cdots +m_{n-2}$ coordinates of $\bft$, and (\ref{eq:estimate-R}) holds for all $x$, it follows that
$$\theta_\eps^\N(E(m_0, m_1,\ldots, m_{n-1}))\le e^{-\rho m_{n-1}}  \theta_\eps^\N (E(m_0,\ldots, m_{n-2})),$$
provided that $m_{n-1}\ge n_0$. Without the last condition on $m_{n-1}$ we have
$$\theta_\eps^\N(E(m_0, m_1,\ldots, m_{n-1}))\le e^{-\rho m_{n-1}} K  \theta_\eps^\N (E(m_0,\ldots, m_{n-2})),$$
where $K= e^{\rho n_0}$. By an easy induction, it follows that
$$\theta_\eps^\N(E(m_0, m_1,\ldots, m_{n-1}))\le e^{-\rho(m_0+m_1+\cdots+ m_{n-1})} K^n .$$
Therefore the left hand side of (\ref{eq:estimate-sumR}) is bounded from above by
\begin{align*}
\sum_{M>R_* n} \sum_{m_0+m_1+\cdots+m_{n-1}=M} e^{-\rho(m_0+m_1+\cdots+ m_{n-1})} K^n =\sum_{M> R_* n} \binom{M-1}{n-1} e^{-\rho M} K^n.
\end{align*}
Provided that we choose $R_*$ large enough, this is exponentially small in $n$. Thus (\ref{eq:estimate-sumR}) follows by redefining $\rho$.
\end{proof}
\begin{prop}\label{prop:gamma-est}
Fix $\gamma> 1-\kappa$. There exist $C>0$ and $\rho>0$ such that
$$\mathbb{P}_\eps(\{m_{\eps^{-\gamma\alpha}}(x,\bft)\ge n\eps^{-\alpha}\varphi(\eps)\})\le Ce^{-\rho n}.$$
\end{prop}
To prove this proposition, we shall use the following elementary observation due to Pliss:
\begin{lemma}\label{lem:pliss}
If $a_0+\dots+a_n\le (n+1) C$ then there exists $k\le n$ so that
$$a_k\le C,a_{k-1}+a_k\le 2C,a_{k-2}+a_{k-1}+a_{k}\le 3C,\dots, a_0+\dots+a_k\le (k+1)C.$$
\end{lemma}

Note that if the conclusion of the lemma holds then for any $\rho\in (0,1)$ and $a_i\ge 0$, then
$$\rho^k a_0+\rho^{k-1} a_1+\dots +a_k
\le
\rho^k(a_0+\dots +a_k)
+\rho^{k-1}(a_1+\dots+a_k)+
\dots + a_k \le $$
$$\le  \sum_{j=1}^{k+1} jC \rho^{j-1} \le \sum_{j=1}^{\infty} jC\rho^{j-1} \le C K(\rho).$$

\begin{proof}[Proof of Proposition~\ref{prop:gamma-est}]
Let $n(x,\bft)$ be the minimal positive integer such that
$$\hR(x,\bft):=\sum_{i=0}^{n(x,\bft)-1} R_i(x,\bft)\le R_* n(x,\bft) \eps^{-\alpha} \varphi(\eps).$$  By Lemma~\ref{lem:R-estimate}, it suffices to
show that  $m_{\eps^{-\gamma\alpha}}(x,\bft) \le \hR(x,\bft)$.

Using Lemma~\ref{lem:pliss} there exists $k\le n(x,\bft)$ with
$$  R_{k-1}(x,\bft)\le R_*  \eps^{-\alpha}\varphi(\epsilon),R_{k-2}(x,\bft)+ R_{k-1}(x,\bft)\le 2R_*  \eps^{-\alpha}\varphi(\epsilon),$$
$$\dots ,R_0(x,\bft)+ \dots + R_{k-1}(x,\bft) \le (k+1)R_* \eps^{-\alpha}\varphi(\epsilon) .$$
Let us show that these inequalities imply that $m_{\eps^{-\gamma}}(x,\bft)\le R_0(x,\bft)+\dots+R_k(x,\bft)$.
Indeed these inequalities  imply, using the notation introduced in equation (\ref{eq:tilde}),
$$\widetilde{ R_k}\le [R_*]+1, \widetilde{R_k}+ \widetilde{R_{k-1}}\le 2([R_*]+1),\dots, \widetilde{R_k}+\dots+\widetilde{R_0}\le (k+1)([R_*]+1).$$
Let $S_i=R_0+\dots+R_{i-1}$. Then by Proposition~\ref{prop:expansionnotbad},
$$\sL^{(R_{i-1})}(f_{\bft}^{S_{i-1}}(x), \sigma^{S_{i-1}}\bft)\le K_2\tR_{i-1}\eps^{-(1-\kappa) \alpha}.$$
Since $$\Lambda^{(S_k-S_i)}(f_{\sigma^{S_i}\bft}^{S_k-S_i}(x), \sigma^{S_i}\bft)\ge \lambda_*^{k-i},$$
we have
$$\sL^{(S_k)}(x,\bft)\le K_2(\tR_{k-1}+\lambda_*^{-\alpha}\tR_{k-2} +\cdots +\lambda_*^{-(k-1)\alpha} \tR_0)\eps^{-(1-\kappa) \alpha}\le \eps^{-\gamma \alpha},$$
provided that $\eps>0$ is small enough.
\end{proof}

\section{Inducing to large scale}
\label{proof:thm-tail-estimate}

Fix $\gamma>0$ and $\kappa\in (\alpha,1)$ such that $2\gamma(1+\alpha)<\alpha$, $\gamma>1-\kappa$ and such that (\ref{eqn:kappa}) holds.
Let $\widehat m(x, \bft)$ denote the minimal positive integer $\widehat m$ for which the following hold:
\begin{itemize}
\item $\widehat m\ge E(x,\bft)$;
\item $\dist (f_\bft^{\widehat m}(x), \PP_0) <\eps^{2\gamma};$
\item
 if $J$ is the component of $f_\bft^{-\widehat m} (B_{\eps^\gamma}(\PP_0))$ which contains $x$, then
$$\sup_{y\in J}\sum_{j=0}^{\widehat m-1} \left(\frac{1}{Df_\bft^{\widehat m-j}(f_\bft^j(y^\pm))}\right)^\alpha\le \eps^{-\gamma\alpha}.$$
\end{itemize}
If $\widehat m$ does not exist then we set $\widehat m(x,\bft)=\infty$.

\begin{lemma}\label{lem:Mtaum}
There exists $K_\sharp >0$ such that the following holds provided that $\eps>0$ is small enough.
For any $x\in I$ and $\bft\in\Omega_\eps^\N$, we have
$$\min (m_{K_\sharp}(x,\bft), \widehat m (x,\bft) )\le m_{\eps^{-\gamma\alpha}}(x,\bft) + \eps^{-\alpha}.$$
\end{lemma}

\begin{proof} Write $m=m_{\eps^{-\gamma\alpha}}(x,\bft)$, $y=f_\bft^m(x)$ and $\bfs=\sigma^m\bft$. If $y\in I_{\Delta_*\eps}$
then  $d(y,P_0)\le \eps^{2\gamma}$ and therefore $\widehat m\le m$ and
the lemma holds. Assume now that $y\not\in I_{\Delta_*\eps}$.
Let $r_0=0$, $r_i=r_i(y,\bfs)$, $i=1,2,\ldots$. Let
$$k_0=\inf\{k\ge 1: \mbox{ either }f_\bfs^{r_k}(y)\in I_{\Delta_* \eps} \mbox{ or } \Lambda^{(r_k)}(y,\bfs)> \eps^{-\gamma}\}.$$ By Proposition~\ref{prop:escape}, for each $1\le k\le k_0$, we have
$r_k-r_{k-1}\le \eps^{-\alpha/(1+\alpha)}$ and
$$\sL_k:=\sL^{(r_k-r_{k-1})}(f_\bfs^{r_{k-1}}(y),\sigma^{r_{k-1}}\bfs) \le K_*,$$
and by Lemma~\ref{lem:constantstau*lambda*}, $$\Lambda_k:=\Lambda^{(r_k-r_{k-1})}(f_\bfs^{r_{k-1}}(y),\sigma^{r_{k-1}}\bfs)\ge \lambda_*>1.$$
Since $\Lambda^{(r_{k_0-1})}(x,\bft)\le \eps^{-\gamma}$ we have $\lambda_*^{k_0-1}\le \eps^{-\gamma}$, hence $k_0\lesssim \log (1/\eps)$ so $r_{k_0}\lesssim \eps^{-\alpha/(1+\alpha)} \log (1/\eps)\ll \eps^{-\alpha}.$
If $\Lambda^{(r_{k_0})}(y,\bfs)> \eps^{-\gamma}$, then
$$\sL^{(m+r_{k_0})}(x,\bft)\le [\Lambda^{(r_{k_0})}(y,\bfs)]^{-\alpha} \sL^{(m)}(x,\bft)+\sum_{i=1}^{k_0} \lambda_*^{-\alpha(i-1)} \sL_{i}\le 1+ K_*\lambda_*^\alpha/(\lambda_*^\alpha-1),$$
so for $K_\sharp$ large enough, $m_{K_\sharp}(x,\bft)\le m+r_{k_0}$ and we are done.
Otherwise, by the definition of $k_0$,  we have $f_{\bfs}^{r_{k_0}}(y)\in I_{\Delta_*\eps}$ and
$$\sL^{(m+r_{k_0})}(x,\bft)\le \lambda_*^{-k_0\alpha} \sL^{(m)}(x,\bft)+\sum_{i=1}^{k_0}\le \lambda_*^{-\alpha}\eps^{-\gamma\alpha}+C\lambda_*^\alpha/(\lambda_*^\alpha-1)<\eps^{-\gamma\alpha},$$
which implies that $\widehat m(x,\bft)\le m+r_{k_0}$, and so we are done again.
\end{proof}

Now let fix $K_\sharp\ge K_*+1$ as above, and define
$$M(x,\bft)=\min \left(m_{K_\sharp} (x,\bft), \widehat{m}(x,\bft)\right).$$
Combining the last lemma with Proposition~\ref{prop:gamma-est} immediately gives us the following corollary.
\begin{coro} \label{cor:Mtautailest}
There exist $C>0$ and $\rho>0$ such that
$$\mathbb{P}_\eps\left\{M(x,\bft)> n \eps^{-\alpha}\varphi(\eps)\right\}\le C e^{-\rho n},$$
provided that $\eps>0$ is small enough.
\end{coro}

\subsection{Proof of Theorem~\ref{thm:tail-estimate}}
Let $$E_1=\{(x,\bft)\in I\times \Omega_\eps^\N: M(x,\bft)< m_{K_\sharp}(x,\bft)\}.$$
Note that $E_1\subset I_{\Delta_*\eps}\times \Omega_\eps^\N$ and that when $(x,\bft)\in E_1$
then $M(x,\bft)=\widehat m(x,\bft)$.
For each $(x,\bft)\in E_1$, let $W_x^\bft$ denote the component of $f_\bft^{-M(x,\bft)}(B_{\eps^{2\gamma}}(\PP_0))$ which contains $x$.

\begin{lemma} 
If $(x,\bft), (x',\bft)\in E_1$, then either $W_x^\bft=W_{x'}^\bft$ or $W_x^\bft\cap W_{x'}^\bft=\emptyset$.
\end{lemma}
\begin{proof}
Write $M=M(x,\bft)$, $M'=M(x',\bft)$, $W=W_x^\bft$, $W'=W_{x'}^\bft$. Arguing by contradiction, assume that $W\cap W'\not=\emptyset$ and $W\not=W'$. Then $M\not=M'$. Assume without loss of generality that $M'>M$. Then $f_\bft^M(x')\not\in B_{\eps^{2\gamma}}(\PP_0)$. Since
$|f_\bft^M(W')|\le |f_\bft^{M'}(W')|=2 \eps^{2\gamma}$, and $f_\bft^M(W)\subset B_{\eps^{2\gamma}}(\PP_0)$ we have
$f_\bft^M(x')\in f_\bft^M(W')\subset B_{3\eps^{2\gamma}}(\PP_0)$. Now let $E=E(f_\bft^M(x'), \sigma^M\bft)$.
Since $M'=\widehat m(x',\bft)$ one has $M'\ge M+E$.
By Proposition~\ref{prop:escape}, $\sL^{(E)}(f_\bft^M(x'),\sigma^M\bft)\le K_*$ and $\Lambda^{(E)}(f_\bft^M(x'),\sigma^M\bft) \gtrsim \eps^{-2\gamma}$. Thus
$$\sL^{M+E}(x',\bft)\le \sL^{(E)}(f_\bft^M(x'),\sigma^M\bft)+ C \eps^{2\gamma\alpha} \sL_{C\eps^{2\gamma}\tau_*}^{(M)} (x',\bft)<K_*+1\le K_\sharp.$$
Thus $m_{K_\sharp} (x',\bft)\le M+E\le  M'< m_{K_\sharp}(x',\bft),$ a contradiction!
\end{proof}

Let $\mathcal{W}=\bigcup_{(x,\bft)\in E_1} W_x^\bft\times \{\bft\}$.  
The lemma above implies that there is a well-defined function $H: \mathcal{W}\to \T\times \Omega_\eps^\N$ such that for any $(x,\bft)\in E_1$, $H(x', \bft)=F^{M(x,\bft)}(x, \bft)$ holds for all $x'\in W_x^\bft$.
For each $n\ge 2$, define $$E_n=\{(x,\bft)\in E_1: H^i(x,\bft)\in E_1 \mbox{ for all }i=0,1,\ldots, n-1\}.$$
Moreover, let $M_n(x,\bft)=M(H^{n-1}(x,\bft))$ for $(x,\bft)\in E_n$. So $M_1(x,\bft)=M(x,\bft)$.

\begin{lemma} \label{lem:Mnest}
There exist $C>0$ and $\rho>0$ such that
\begin{enumerate}
\item $\mathbb{P}_\eps(E_n)\le Ce^{-\rho n}$.
\item For each $n, m\ge 1$,
$$\mathbb{P}_\eps(\{(x,\bft)\in E_{n}: M_n(x,\bft)\ge m\eps^{-\alpha}\varphi(\eps)\})\le C e^{-\rho m}.$$
\end{enumerate}
\end{lemma}
\begin{proof} For each $(x,\bft)\in\mathcal{W}_n:=H^{-n}(\mathcal{W})$, let $s_n(x,\bft)$ be such that $H^n(x,\bft)=f_\bft^{s_n(x,\bft)}(x,\bft)$ and let $J_{n}^\bft(x)$ denote the component of $\mathcal{W}_n^\bft$ which contains $x$.
Note that there exists $\lambda>1$ such that $\mathcal{W}\subset B_{\lambda^{-1}\eps^{2\gamma}}(\PP_0)$. So $s_n(\cdot,\bft)$ is constant in each of these intervals $J_n^\bft(x)$, $f_\bft^{s_n(x,\bft)}(J_n^\bft(x))$ is a component of $B_{\eps^{2\gamma}}(\PP_0)$, and moreover,
$\sum_{j=0}^{s_n(x,\bft)-1}|f_\bft^j(J_n^\bft(x))|^\alpha$ is bounded from above by a constant.  Thus by part (3) of Lemma~\ref{lem:distortionlemma}, there exists a constant $C>0$ such that for each set $A\subset B_{\lambda^{-1}\eps^{2\gamma}}(\PP_0)$, each $J=J_n^\bft(x)$ and $s=s_n(x,\bft)$, we have
\begin{equation}\label{eqn:distortionWn}
\frac{|J\cap f_\bft^{-s}(A)|}{|J|}\le \frac{C|A|}{C|A|+\eps^{2\gamma}}.
\end{equation}

Taking $A=\mathcal{W}_{1}^{\sigma^s\bft}$ we obtain that $|J\cap\mathcal{W}_{n+1}^\bft|/|J|$ is bounded away from one. Thus $|\mathcal{W}_{n+1}^\bft|/|\mathcal{W}_n^\bft|$ is bounded away from one. So $\mathbb{P}_\eps(\mathcal{W}_{n+1})/\mathbb{P}_\eps(\mathcal{W}_n)$ is bounded away from one. Therefore $\mathbb{P}_\eps(E_n)\le \mathbb{P}_\eps(\mathcal{W}_n)$ is exponentially small. This proves the first statement.

For the second statement, fix $m, n$, let $X=\{(x,\bft)\in E_n: M_n(x,\bft)\ge m\eps^{-\alpha}\varphi(\eps)\}$ and $Y=\{(y,\bfs): M(y,\bfs)\ge m\eps^{-\alpha}\varphi(\eps)\}$.
For each $s\ge 1$, let $\mathcal{V}(s)=\{(x,\bft)\in\mathcal{W}_{n-1}: s_{n-1}(x,\bft)=s\}$. 
Taking $A=Y^{\sigma^s \bft}$ in (\ref{eqn:distortionWn}) gives
$$\frac{|J\cap X^\bft|}{|J|}\le C |Y^{\sigma^s\bft}|\eps^{-2\gamma}, \mbox{ for each } J=J_{n-1}^\bft(x), s=s_{n-1}(x,\bft).$$
Thus $$|(X\cap \mathcal{V}(s))^\bft|\le C \eps^{-2\gamma} \sum_{s=1}^\infty |Y^{\sigma^s\bft}| |(\mathcal{V}(s))^\bft|.$$
Since the set $(\mathcal{V}(s))^\bft$ depends only on the coordinates $t_0, t_1,\cdots, t_{s-1}$ of $\bft$, by Fubini's theorem, we obtain
\begin{align*}
\mathbb{P}_\eps(X)& =\sum_{s=1}^\infty \mathbb{P}_\eps(X\cap\mathcal{V}(s))\\
 & \le C \eps^{-2\gamma}\sum_{s=1}^\infty \int_{\Omega_\eps^\N} |Y^{\sigma^s\bft}| |(\mathcal{V}(s))^\bft| d\theta_\eps^\N(\bft)\\
& \le C \eps^{-2\gamma} \sum_{s=1}^\infty \int_{\Omega_\eps^\N} |Y^{\sigma^s\bft}|d\theta_\eps^\N(\bft)\int_{\Omega_\eps^\N} |(\mathcal{V}(s))^\bft| d\theta_\eps^\N(\bft)\\
& \le C \eps^{-2\gamma} \sum_{s=1}^\infty \int_{\Omega_\eps^\N} |Y^{\bft}|d\theta_\eps^\N(\bft) \mathbb{P}_\eps(\mathcal{V}(s))\\
& = C\eps^{-2\gamma}  \mathbb{P}_\eps(\mathcal{W}_{n-1}) \mathbb{P}_\eps (Y)\\
& \le C \eps^{-2\gamma} \mathbb{P}_\eps(B_{\eps^{2\gamma}}(\PP_0))\mathbb{P}_\eps(Y)\lesssim \mathbb{P}_\eps(Y).
\end{align*}
By Corollary~\ref{cor:Mtautailest} the second statement of this lemma follows.
\end{proof}

\begin{prop}\label{prop:tailestmlarge}
There exist constants $K_@$, $C>0$ and $\rho>0$ such that
$$\mathbb{P}_\eps\left(\{ (x,\bft)\in I\times \Omega^\N_\eps; m_{K_@} (x,\bft) > n\eps^{-\alpha} \varphi(\eps)\}\right)\le Ce^{-\rho \sqrt{n}}.$$
\end{prop}
\begin{proof} Note that if $(x,\bft)\in E_k\setminus E_{k+1}$, then $(x,\bft),\dots,H^{k-1}(x,\bft)\in E_1$ and
$H^k(x,\bft)\notin E_1$
and it follows that  $M_1(x,\bft)+M_2(x,\bft)+\cdots+ M_k(x,\bft)$ is a
$K_@$-expansion time of $(x,\bft)$, where $K_@$ is a constant. So
if $m_{K_@} (x,\bft) > n\eps^{-\alpha} \varphi(\eps)$, then either $M_1(x,\bft)> n\eps^{-\alpha}\varphi(\eps)$, or $(x,t)\in E_k$ for some $k\ge \sqrt{n}$, or
there exists $1\le k<\sqrt{n}$ such that $M_k(x,\bft)\ge \sqrt{n} \eps^{-\alpha}\varphi(\eps)$. Using Corollary~\ref{cor:Mtautailest} for the first case and Lemma~\ref{lem:Mnest} for the remaining two cases, we obtain the desired estimate.
\end{proof}

\begin{proof}[Proof of Theorem~\ref{thm:tail-estimate}]
Fix $\tilde\alpha\in (1, \alpha)$. Take $\beta=\sqrt{\alpha \tilde{\alpha}}$ and choose $\kappa\in (0,1)$ such that
$\kappa^2 \tilde \alpha >\alpha$. Let $K_0=\max (K_*, K_@)$ and write $A_m=\mathbb{P}_\eps(\{(x,\bft)\in I\times \Omega^\N_\eps; m_{K_0}(x,\bft)\ge m\})$.
If $m\le \eps^{-\kappa\alpha}$, then by Corollary~\ref{cor:tailmsmall}, we have $A_m\lesssim m^{-1/\alpha}$.
If $\eps^{-\kappa \alpha}\le m\le \eps^{-\beta}$, then
$A_m\lesssim \eps^{\kappa}\le m^{-\kappa/\beta}\le m^{-1/\tilde \alpha}$ (where the first inequality follows from
Proposition~\ref{prop:escape} and from $m\ge \eps^{-\kappa\alpha}$).
If $m>\eps^{-\beta}$ then
let $n$ be the maximal integer so that
$m\ge n\eps^{-\alpha}\varphi(\eps)$. So $n\asymp m\eps^{\alpha}/\varphi(\eps)$.
Since $m>\eps^{-\beta}$ and by Proposition~\ref{prop:tailestmlarge}, we therefore have
$A_m \lesssim e^{-\rho \sqrt{n}}\lesssim m^{-1/\alpha}$, provided that $\eps>0$ is small enough.
\end{proof}


\section{Strong stochastic stability}
\label{sec:final}

\subsection{Nice sets}
An interval $V$ is a {\em nice set} for the deterministic system $f\colon \T\to \T$
if $f^n(x)\notin V$ for each $x\in \partial V$ and all $n>0$. This notion has played
a crucial role in the setting of one-dimensional dynamical systems, because
it implies that if $x\in V$ and $n>0$ is minimal so that $f^n(x)\in V$,
then the component of $f^{-n}(V)$ containing $x$ is contained in $V$.
This means that each component of the domain of the first return map $\mathcal{R}_V$ to $V$
is contained in $V$ and that its boundary points are mapped into boundary points of $V$.

In the deterministic case it is often helpful to construct nice sets around  special points, such
as critical points or parabolic periodic points.
In our case we will construct nice sets around the fixed points.
Rivera-Letelier \cite{R} realised that it very advantageous to have the property
which in our setting states that there exists $K>1$ so that
for  each $\delta>0$ small there exists a nice interval $V$ with
$B_\delta(p_0)\subset V\subset B_{K\delta}(p_0)$.  This property was also crucially used in
\cite{BRSS}. In the case of random dynamical systems, nice sets and the analogous property
were also used and  established in \cite{Shen}. It implies that one can induce to a
 Markov setting while having extension {\lq}space{\rq}.
In our setting nice sets are defined as follows:

\begin{definition}
A {\em nice set for $\eps$-perturbations} is a measurable subset $\textbf{V}$ of $\T\times \Ome$ with the following properties:
\begin{enumerate}
\item for each each $\bft\in \Ome$, $\textbf{V}^\bft$ is an open set containing $\PP_0$ such that each component of $\textbf{V}^\bft$ intersects $\PP_0$ exactly at one point;
\item for each $\bft\in \Ome$, each $x\in \partial \textbf{V}^\bft$ and each $n\ge 1$, we have
$$f_\bft^n(x)\not\in \textbf{V}^{\sigma^n\bft}.$$
\end{enumerate}
\end{definition}

\begin{lemma}\label{lem:niceset}
There exists $\delta_0>0$ such that for each $\delta\in (0,\delta_0)$,  there exists a nice set $\textbf{V}$ for $\eps$-perturbations such that for each $\bft\in \Ome$ and each $p_0\in \PP_0$, we have
$$B_\delta(p_0)\subset \textbf{V}^\bft(p_0)\subset B_{2\delta}(p_0),$$
provided that $\eps>0$ is small enough.
\end{lemma}
\begin{proof}
The proof is similar to the one given in \cite[Proposition 5.8]{Shen}. It suffices to show for small $\delta>0$, the following holds provided that $\eps>0$ is small enough: if $\delta\le \dist(x, \PP_0)<2\delta$, $\bft\in\Omega_\eps^\N$ and $f_\bft^m(x)\in B_{2\delta} (\PP_0)$, then the component $W$ of $f_\bft^{-m}(B_{2\delta}(\PP_0))$ which contains $x$ satisfies either $W\subset B_{2\delta}(\PP_0)$ or $W\cap \PP_0=\emptyset$. Indeed, otherwise, there would exist $y\in W$ such that
$\dist(f_\bft^m(y),\PP_0)<\dist(y,\PP_0)\le 2\delta$, so $m\ge E(y,\bft)=:E$. However, by Proposition~\ref{prop:escape}, $\Lambda^{(m)}(y,\bft)\ge \Lambda^{(E)}(y,\bft)\ge (2K_*\delta)^{-1}>4$ (provided that $\delta>0$ is small enough). This would imply that $|Df_\bft^m(z^+)|>4$ for all $z\in W$, hence $|W|< \delta$, a contradiction.
\end{proof}

\subsection{Inducing to the nice set}

The following theorem allows us to consider an induce transformation
to the nice set $\textbf{V}$.

\begin{theorem}\label{thm:tail-estimate-nice}
Fix $\tilde{\alpha}\in (1,\alpha)$. For each $\delta>0$ small there exist constants $K$, $C$ and $\eps_0>0$
so that for each $\eps\in (0,\eps_0)$ there exists a nice set $V$ for $\eps$-perturbations
so that
\begin{enumerate}
\item $B_\delta(p_0)\times \Ome \subset \mathbf{V} \subset B_{2\delta}(p_0)\times \Ome$
for each $x\in \mathbf{V}^\bft$;
\item for $\mathbb{P}_\eps$-almost each $(x,\bft)\in \textbf{V}$ there exists a positive integer $m=m_{\mathbf{V}}<\infty$
such that
$$f_\bft^m(x)\in \textbf{V}^{\sigma^m\bft}, \,\, \sL^{(m)}(x, \bft)\le K, \mbox{ and }\Lambda^{(m)}(x,\bft)\ge \lambda_*;$$
\item $\mathbb{P}_\eps(\{(x,\bft)\in \mathbf{V}; m_{\mathbf{V}}(x,\bft)>m\})\le C_0 m^{-1/\tilde{\alpha}}$.
\end{enumerate}
\end{theorem}
\begin{proof}
Let $K_0$ and $C_0$ be given by Theorem~\ref{thm:tail-estimate} and let $\delta_0$ be given by Lemma~\ref{lem:niceset}. Fix a constant $\delta$ such that $0<\delta\in \min (\delta_0, \tau_*/4)$ and assume $\eps>0$ is small. So there exists a nice set $\mathbf{V}$ for $\eps$-perturbations such that $B_\delta(\PP_0)\subset \mathbf{V}^\bft\subset B_{2\delta}(\PP_0)$ holds for all $\bft\in \Omega_\eps^\N$.
For each $(x,\bft)$, let $s(x,\bft)=\inf\{s\in \N: F^s(x,\bft)\in \mathbf{V}\}$.  Since $f$ is uniformly expanding outside $B_{\delta}(\PP_0)$ there exist $C_1>0$ and $\rho>0$ such that
$$\mathbb{P}_\eps\left(\{s(x,\bft)\ge s\}\right)\le C_1 e^{-2\rho s}.$$
Note that $s(x,\bft)$ is an essential return time of $(x,\bft)$ into $I$ and all previous essential returns lie outside $I_{\Delta_* \eps}$, so by Lemma~\ref{lem:distgeneral} and Lemma~\ref{lem:slrec}, we obtain that
$\sL^{s(x,\bft)}(x,\bft)$ is bounded from above by a constant.
Define $$m_{\mathbf{V}}(x,\bft)= m_{K_0}(x,\bft) + S(x,\bft), \mbox{ where } S(x,\bft)=s(F^{m_{K_0}(x,\bft)}(x,\bft)).$$
Then by Lemma~\ref{lem:slrec} again, we obtain that $\sL^{(m_{\mathbf{V}(x,\bft)})}(x,\bft)$ is bounded from above by a constant. Thus the second item of the theorem holds for some $K$.

It remains to prove the third item. Let $d$ be the mapping degree of $f:\T\to \T$. So $f_\bft^n: \T\to \T$ is a $d^n$-to-$1$ covering map. Write
$$X_m=\{(x,\bft)\in \mathbf{V}: m_{\mathbf{V}}(x,\bft)>m, m_{K_0}(x,\bft)> \rho (\log d)^{-1} S(x,\bft))\},$$
and $$Y_m=\{(x,\bft)\in\mathbf{V}:  m_{\mathbf{V}}(x,\bft)>m, m_{K_0}(x,\bft)\le \rho (\log d)^{-1} S(x,\bft))\}.$$
For $(x,\bft)\in X_m$, $m_{K_0}(x,\bft)\gtrsim m$, hence by Theorem~\ref{thm:tail-estimate} we obtain $\mathbb{P}_\eps(X_m)=O(m^{-1/\tilde{\alpha}})$. Given positive integers $k,s$ with $k\le \rho(\log d)^{-1}s$, we have
$$\mathbb{P}_\eps (\{(x,\bft)\in Y_m: m_{K_0}=k, S=s\})\le  d^k C_1 e^{-2\rho s}\le C_1 e^{-\rho s}.$$  For $(x,\bft)\in Y_m$,
$S(x,\bft)\ge \rho_1 m$, where $\rho_1>0$ is a constant, so
$$\mathbb{P}_\eps (Y_m)\le \sum_{s\ge \rho_1 m} C_2 s e^{-\rho s}=O(m^{-1/\tilde{\alpha}}).$$
Thus
$$\mathbb{P}_\eps(\{(x,\bft)\in V; m_{\mathbf{V}}(x,\bft)>m\})\le \mathbb{P}_\eps(X_m)+\mathbb{P}_\eps(Y_m)=O( m^{-1/\tilde{\alpha}}).$$
\end{proof}


To prove the first part of the Main Theorem we first state the following

\begin{definition}
A Borel set $E\subset \T$ is called {\em almost forward (resp. backward) invariant} for $\eps$-perturbations if
$f_\bft(E)\setminus E$ (resp. $f_\bft^{-1}(E)\setminus E$) has zero Lebesgue measure for a.e. $\bft\in \Ome$.
A stationary measure is called {\em ergodic} if for each Borel set $E$ which is almost
completely invariant, one has $\mu_\eps(E)=$ or $\mu_\eps(E)=1$.
\end{definition}

\begin{prop}
For each $\epsilon>0$, the random dynamical system has a unique stationary measure. This stationary measure is ergodic and absolutely continuous.
\end{prop}
\begin{proof}
The proof of these statements is fairly standard, and can be found for example in Lemmas 3.2 and 3.4 in
\cite{Shen}.
\end{proof}

\subsection{Perron-Frobenius operator}
As usual, we will use the Perron-Frobenius operator.
Given $J\subset \T$, and $f_\bft$
we define
$$\mathcal{L}_{J,n,\bft}(x)=\sum_{f^n_\bft(y)=x, y\in J} \dfrac{1}{|Df^n_\bft( y^+)|} \mbox{ and }
\widehat{\mathcal{L}_{J,n,\bft}}=\dfrac{1}{|J|}\mathcal{L}_{J,n,\bft}.$$
We should remark that $\mathcal{L}_{J,n,\bft}$ is the density function of the absolutely continuous measure
$(f_\bft^n)_*(\Leb|J)$ and has support $f^n_\bft(J)$.
In other words, $\mathcal{L}_{J,n,\bft}$ is equal to $(\mathcal{PF})^n(1_J)$ where $\mathcal{PF}$ is the Perron-Frobenius
operator. We also note that $\widehat{\mathcal{L}_{J,n,\bft}}$ is the density of the push forward of the relative measure on
$J$, so the integral over $\T^1$ of this density is equal to $1$.

\subsection{Compact subsets of $L^1$}\label{subsec:compactinL1}
Let $L^1(\T)$ be the Banach space of $L^1$ functions w.r.t. the Haar measure $\Leb$ on $\T$
and let $||\phi||_1$ denote the $L^1$ norm of $\phi$.
In this subsection we will  establish a suitable compact subset of $L^1$.

The following lemma is well-known. See for example~\cite[Theorem IV.8.20]{DS}.

\begin{lemma}[Precompact sets in $L^1$]\label{lem:precompact}
A bounded set $\mathcal{K}$ of functions in $L^1$ is pre-compact if and only if
$||\psi(x+h)-\psi(x)||_{L^1}$ tends to zero as $h\to 0$ uniformly
for all $\psi\in \mathcal{K}$.
\end{lemma}

Let as before $\PP_*=\{p\in \PP: Df(p^-)>Df(p^+)\}$ and let $$\F=\{Df(p^-)/Df(p^+): p\in \PP_*\}$$ which is a finite subset of $(1,\infty)$.
Given a constant $C>0$ and an interval $W=(a,b)\subset \T$, let $\mathcal{E}_{C}(W)$ denote the space of maps
$\varphi: \T \to (0,\infty)$ for which there exists finitely many points
$$a=a_0<a_1<\cdots< a_n=b$$
such that
\begin{itemize}
\item $\varphi\equiv 0$ in $\T\setminus W$;
\item for each $0\le i<n$, there exists $C_i>0$ such that
$$|\log\varphi(x)-\log\varphi(y)|\le C_i \left(\frac{|x-y|}{a_{i+1}-a_i}\right)^\alpha, \,\, a_i< x<y<a_{i+1},$$
and such that $$\sum_{i=0}^{n-1} C_i\le C;$$
\item for each $1\le i<n$, $\varphi(a_i)=\varphi(a_{i}^+)= \lambda \varphi(a_{i}^-)$ for some $\lambda\in \F$.
\end{itemize}
For $M>0$, let ${\sE}_C^M$ denote the set of all $L^1$ functions $\varphi: \T\to [0,\infty)$ for which there exists an open interval $W\subsetneq\T$ such that
$\varphi\in \sE_C(W)$ and $\varphi(x)\le M$ for all $x\in \T$.

\begin{lemma}\label{lem:cptness1}
\begin{enumerate}
\item If $\varphi\in \sE_{K}(W)$ for some open interval $W\subset \T$, then for any $x, y\in W$ with $x<y$ we have
$$\varphi(x)\le e^K \varphi(y).$$ Moreover, if $\varphi$ has at least $N$ points of discontinuity in $(x,y)$, then
$$\varphi(x)\le e^K \lambda_0^{-N} \varphi(y),$$ where $\lambda_0=\inf\F>1$.
\item For each $K>0$ and $M>0$, ${\sE}_K^M$ is pre-compact in $L^1(\T)$.
\end{enumerate}
\end{lemma}
\begin{proof} (1) Note that if $a_i\le x<y\le a_{i+1}$, $\varphi(x)=\varphi(x^+)\le e^{K_i}\varphi(y^-)\le e^K\varphi(y)$. For general $x, y\in W$ with $x<y$, let $i_0$ be maximal such that $x\ge a_{i_0}$ and let $i_1$ be minimal such that
$y< a_{i_1}$, we have
\begin{align*}
\varphi(x)\le e^{K_{i_0}}\varphi(a_{i_0+1}^-)\le
e^{K_{i_0}}\lambda_0^{-1} \varphi(a_{i_0+1})
\le  e^{K_{i_0}+K_{i_0+1}}\lambda_0^{-2} \varphi(a_{i_0+2})\\
\le \cdots \le e^{\sum_{i=i_0}^{i_1-2} K_i}\lambda_0^{-i_1+i_0+1}\varphi(a_{i_1-1}) \le e^K\lambda_0^{-i_1+i_0+1} \varphi(y).
\end{align*}

(2) For any integer $N\ge $, let ${\sE}_K^M(N)$ denote the subset of ${\sE}_K^M$ consisting of functions $\varphi$ with exactly $N$ points of discontinuity. Using Lemma~\ref{lem:precompact} it is easy to see that for each $N\ge 0$,
${\sE}_K^M(N)$ is pre-compact in $L^1(\T)$. To complete the proof, it suffices to show that for each $\eta>0$, there exists $N(\eta)\ge 0$ such that for any $\varphi\in {\sE}_K^M$ there exists $\varphi_N\in {\sE}_K^M(N)$ such that $\|\varphi-\varphi_N\|_1\le \eta$. But this follows easily from the last statement of part (1) of this lemma.
\end{proof}

\begin{lemma} \label{lem:distortionlemma}
For each $K>0$ and $\rho>0$ there exist $C=C(K)>0$ and $M=M(K,\rho)>0$ such that the following holds. Let $J\subset \hJ$ be open intervals, $m\ge 1$, $\bft\in\Ome$ such that
\begin{enumerate}
\item [(i)] $f_\bft^m: \hJ\to f_\bft^m(\hJ)$ is injective;
\item [(ii)] $\sum_{j=0}^{m-1}|f_\bft^j(\hJ)|^\alpha\le K$;
\item [(iii)] $f_\bft^m(\hJ)\setminus f_\bft^m(J)$ has a component to the right of $f_\bft^m(J)$ which has length $\ge \rho |f_\bft^m(\hJ)|$.
\end{enumerate}
Then
\begin{enumerate}
\item [(1)] $\mathcal{L}_{\hJ, m, \bft} \in \sE_C(f_\bft^n(\hJ)).$
\item [(2)] $|\hJ|^{-1} \mathcal{L}_{J, m, \bft} \in {\sE}_C^M.$
\item [(3)] for any $X\subset J$, $$\frac{|X|}{|\hJ|}\le M \frac{|f_\bft^m(X)|}{|f_\bft^m(\hJ)|}.$$
\end{enumerate}
\end{lemma}
\begin{proof} (1) Let $A=\{x\in \hJ: f_\bft^j(x)\in \PP_* \mbox{ for some } 0\le j<m\}$. This is a finite set (maybe empty) and label its elements as $a_1< a_2< \cdots< a_{n-1}$. Let $a_0<a_n$ be the endpoints of $\hJ$ and for each $0\le i\le n$, let $a_i'=f_\bft^m(a_i)$. Then for each $0\le i<m$, $f_\bft^m$ maps $\hJ_i:=(a_i, a_{i+1})$ diffeomorphically onto $\hJ_i':=(a_i', a_{i+1}')$.  Moreover, there exists a constant $C_0>0$ such that
$$\Dist (f_\bft^m |\hJ_i)\le C_0 K_i\le C_0 K, \mbox{ where }K_i=\sum_{j=0}^{m-1} |f_\bft^j(\hJ_i)|^\alpha.$$
Let $\varphi$ be the inverse of $f^m\colon \hJ\to \hJ'$ and
for each $x', y'\in \hJ_i'$, letting $x, y\in \hJ_i$ be their images under $\varphi$.  We have
\begin{align*}
\begin{aligned}
|\log \varphi(x')- \log \varphi(y')|  &=|\log Df_\bft^m(x)-\log Df_\bft^m (y)|  \le C_0 \sum_{j=0}^{m-1} |f_\bft^j(x)-f_\bft^j(y)|^\alpha
 \\
 & \le C_0 e^{C_0K\alpha} \frac{|x'-y'|^\alpha}{|\hJ_i'|^\alpha}\sum_{j=0}^{n-1} |f_\bft^j(\hJ_i)|^\alpha  \le C_1 K_i \frac{|x'-y'|^\alpha}{|\hJ_i'|^\alpha} ,
 \end{aligned}
\end{align*}
where $C_1= C_0 e^{C_0K\alpha}$. Note that $\sum_{i=0}^{n-1} C_1 K_i\le C_1 K=:C$.  Clearly for each $1\le i<n$, $\varphi(a_i')=\varphi(a_i'^+)=\lambda \varphi(a_i'^-)$ holds for some $\lambda\in \F$.
(It is possible that there exists a point $x\in \hJ_i$ so that several
of the points $x,f_\bft(x),\dots,f_\bft^{m-1}(x)$ are in $P_0$, in which one can take several of the points
$a_i$ to coincide.)
Thus
$\varphi\in \sE_{C}(f_\bft^m(\hJ))$.

(2) Let $R$ denote the component of $\hJ\setminus J$ which lies to the right of $J$ and let $R'=f_\bft^m(R)$. Let $\psi= |\hJ|^{-1}\mathcal{L}_{\hJ, m,\bft}$. Then $\psi\in \sE_C$. By Lemma~\ref{lem:cptness1}, it follows that $\psi(x)\le e^C\psi(y)$ for all $x\in f_\bft^m(J)$ and any $y\in R'$. Since $\|\psi\|_1=1$, it follows that there exists $M=M(\rho,K)$ such that $\psi(x)\le M$ for all $x\in f_\bft^m(J)$.
Since $|\hJ|^{-1} \mathcal{L}_{J, m,\bft}= \psi 1_{f_\bft^m(J)}$, the statement follows.

(3) By Part (1) of the lemma, and Lemma~\ref{lem:cptness1}, for each $x\in X$ and $y\in R$, we have
$Df_\bft^n(x)\ge e^{-C} Df_\bft^n(y).$ Thus
$$\frac{|f_\bft^m(X)|}{|X|}\ge e^{-C}\frac{|f_\bft^m(R)|}{|R|}\ge e^{-C}\rho \frac{|f_\bft^m(\hJ)|}{|\hJ|},$$
which implies the desired estimate by taking $M=e^C/\rho$.
\end{proof}

\subsection{Proof  of the Main Theorem}

Now we will combine the previous compact results, with Theorem~\ref{thm:tail-estimate-nice}
to obtain the proof of the Main Theorem.

\begin{proofof}{Main Theorem}
Statements (1) and (2) in the Main Theorem follow exactly as the first part of the proof
of the Main Theorem in Section 3.1 of \cite{Shen}. The proof of Statements (3) and (4) of the Main Theorem
is slightly different, and therefore we give the argument here.

Let $\mathbf{V}$ be a nice set for the $\eps$-perturbations given by Theorem~\ref{thm:tail-estimate-nice}, let $\mathbf{U}=\{(x,\bft)\in V: m_{\mathbf{V}}(x,\bft)<\infty\}$, and
let $G:\textbf{U}\to \mathbf{V}$ denote the map $(x,\bft)\to F^{m_{\mathbf{V}}(x,\bft)}(x,\bft)$.
Let $Z=(-\delta,\delta)$ and let
$$\phi_n(x)=\phi_{n,\eps}(x)=\int_{\bft} \mathcal{L}_{Z,n,\bft}(x) \, d\theta_\eps^\N.$$

\noindent
{\bf Claim:}  It suffices to show that there exists a compact subset $\mathcal K$ of $L^1$ which
does not depend on $\eps$ and $n$ and so that $\phi_n\in \mathcal K$ for all $n$ and all $\eps>0$ small.
{\bf Proof of Claim:} The assumption of the claim implies that there is a compact subset $\mathcal{K}_0$ (the convex hull of $\mathcal{K}$) of $L^1$, such that for each $n=1,2,\ldots$ and each $\eps>0$ small enough, we have $\dfrac{1}{n}\sum_{i=0}^{n-1}\phi_{i,\eps}(x)\in \mathcal{K}_0$.  Since $\mu_\eps$ is the weak star limit of $\dfrac{1}{n}\sum_{i=0}^{n-1}\phi_{i,\eps}(x)dx$ as $n\to\infty$, it follows that $\mu_\eps=\psi_\eps(x) dx$ for some $\psi_\eps\in \mathcal{K}_0$.
Since any limit of $\psi_\epsilon$ as $\epsilon\to 0$ is 
the density  of
an absolutely continuous invariant measure of $f$,
and since $f$ has at most one absolutely continuous invariant probability measure,
it follows that  $\psi_\eps$ converges in
$L^1$ as $\epsilon\to  0$. Thus it follows that
$f$ has an absolutely continuous invariant measure $\psi dx$ and that the density of the measure
$\mu_\eps$ converges in $L^1$ to $\psi$ as $\epsilon\to 0$, thus concluding
the proof of the claim.

By considering subsequences, it even suffices to find for each $\eta>0$
a compact set $\mathcal K_\eta$ of $L^1$ so that each function $\phi_n$
can be written as
$\phi_n=\phi_n^0+\phi_n^1$
where $||\phi_n^0||_1\le \eta$ and $\phi_n^1\in {\mathcal K}_\eta$.
In order to do this, we will choose a suitable $M$ depending on $\eta$.

For $k\ge 0$, let $\mathbf{U}_k=G^{-k}(\textbf{V})$ and $\mathbf{V}_k=\{(x,\bft)\in \mathbf{V}: m_{\mathbf{V},K}\ge k\}$. For each $\bft\in\Ome$, let us label the components of $\mathbf{U}_k^\bft$ as
$J_{k,i}^{\bft}$ and let $s_{k,i}^{\bft}$ be the integer so that
$$G^k(y,\bft)=F^{s_{k,i}^\bft}(y,\bft)\mbox{ for }y\in J_{k,i}^\bft.$$
For each $k,i$, let $\hJ_{k,i}^\bft$ denote the
component of $f_\bft^{-s_{k,i}^\bft}(B_{\tau_*/2}(\PP_0))$ which contains $J_{k,i}^\bft$. Note that
\begin{equation}\label{eqn:hJlengthsum}
\sum_{j=0}^{s_{k,i}^\bft-1} |f_\bft^j(\hJ_{k,i}^\bft)|^\alpha\le \sL^{s_{k,i}^{\bft}}(y,\bft)\le \frac{K\lambda_*^\alpha}{\lambda_*^\alpha-1}\tau_*^\alpha.
\end{equation}

Fix $n\ge 1$. For each $0\le m< n$, let $\mathcal{H}_m^\bft=\{J_{k,i}^\bft: s_{k,i}^\bft =m\},$ and for each $J_{k,i}^\bft\in\mathcal{H}_m^\bft$, let
$J_{k,i,j}^\bft$ denote the components of $J_{k,i}^\bft\cap f_\bft^{-m}(\mathbf{U}^{\sigma^m\bft})$, let $m_{k,i,j}^\bft$ denote the value of $m_{\mathbf{V}}$ on the interval
$f_\bft^m(J_{k,i,j}^\bft)$,
and let $\mathcal{X}_{k,i}^\bft$ denote the collection of all $J_{k,i,j}^\bft$ for which $m_{k,i,j}^\bft\ge n-m$. Moreover, for each $M>0$, let $\mathcal{X}_{k,i}^\bft(M)$ denote the sub-collection of
$\mathcal{X}_{k,i}^\bft$ consisting of those $J_{k,i,j}^\bft$'s for which $m_{k,i,j}^\bft\ge M$. When $m\le n-M$, $\mathcal{X}_{k,i}^\bft(M)=\mathcal{X}_{k,i}^\bft$ but when $m>n-M$ the former set may be strictly smaller.
Furthermore, let $X_{k,i}^\bft$ (resp. $X_{k,i}^\bft(M)$) denote the union of the elements of $\mathcal{X}_{k,i}^\bft$ (resp. $\mathcal{X}_{k,i}^\bft(M)$).

By part (3) of Lemma~\ref{lem:distortionlemma}, (\ref{eqn:hJlengthsum}) implies that there exists $C_0>0$ such that whenever $s_{k,i}=m\le n-M$,
\begin{equation}
\frac{|X_{k,i}^\bft(M)|}{|\hJ_{k,i}^\bft|}=\frac{|X_{k,i}^\bft|}{|\hJ_{k,i}^\bft|}\le C_0 \frac{|\mathbf{V}_{n-m}^{\sigma^m\bft}|}{|f_\bft^m(\hJ_{k,i}^\bft)|}=C_0\tau_*^{-1} |\mathbf{V}_{n-m}^{\sigma^m\bft}|,
\end{equation}
and whenever  $n-M<s_{k,i}=m<n$,
\begin{equation}
\frac{|X_{k,i}^\bft(M)|}{|\hJ_{k,i}^\bft|}\le C_0 \frac{|\mathbf{V}_{M}^{\sigma^m\bft}|}{|f_\bft^m(\hJ_{k,i}^\bft)|}=C_0\tau_*^{-1} |\mathbf{V}_{M}^{\sigma^m\bft}|.
\end{equation}
Putting $X^{m,\bft}(M)=\bigcup_{s_{k,i}^\bft=m} X_{k,i}^\bft(M)$ for each $\bft$, since the intervals $\hJ_{k,i}^\bft\in \mathcal{H}_m^\bft$ are pairwise disjoint, it follows that
there exists a constant $C_1>0$ and a compact subset $\mathcal{K}$ of $L^(\T)$ such that
\begin{equation}
|X^{m,\bft}(M)|\le C_1 |\mathbf{V}_{m'}^{\sigma^m\bft}|, \mbox{ where } m'=\max(n-m, M).
\end{equation}
Choose a large positive integer $M$ such that $\sum_{i=M}^\infty \mathbb{P}_\eps (\mathbf{V}_i) + M\mathbb{P}_\eps(\mathbf{V}_M) \le \eta/C_1$, and let
\begin{equation*}
\varphi_n^0=\sum_{m=0}^{n-1} \int_\bft \mathcal{L}_{X^{m,\bft}(M), n, \bft} d\theta_\eps^\N\mbox{ and }\varphi_n^1=\varphi_n-\varphi_n^{0}.
\end{equation*}
Then
$$\|\varphi_n^0\|_1\le \sum_{m=0}^{n-1}\int_\bft |X^{m,\bft}(M)|d\theta_\eps^\N\le \eta.$$

It  remains to show that $\varphi_n^1$ defined above belongs to a compact subset $\mathcal{K}(\eta)$ of $L^1(\T)$.
Note that for a.e. $\bft$, the intervals $J_{k,i,j}^\bft$ with $m_{k,i,j}^\bft+s_{k,i}^\bft\ge n$ and $0\le s_{k,i}^\bft<n$, form a partition of  $\mathbf{V}^\bft$, up to a set of Lebesgue measure zero. Thus, for $n-M<m<n$, putting $Y_m^\bft$ to be the union of the intervals $J_{k,i,j}^\bft$ with $s_{k,i}^\bft=m$ and $n-m\le m_{k,i,j}^\bft< M$, we have
$$\varphi_n^1=\sum_{m=n-M+1}^{n-1}\int_\bft \mathcal{L}_{Y_m^\bft, n, \bft} d\theta_\eps^\N.$$
By part (2) of Lemma~\ref{lem:distortionlemma} and (\ref{eqn:hJlengthsum}),
$|\hJ_{k,i}^\bft|^{-1}\mathcal{L}_{J_{k,i,j}^\bft, s_{k,i}^\bft,\bft}$ is contained in $\sE_C^M$ for some constants $C, M>0$. By Lemma~\ref{lem:cptness1}, $\sE_C^M$ is pre-compact in $L^1(\T)$. For each $(k,i)$ the number of intervals $J_{k,i,j}^\bft$ is uniformly bounded from above by a constant depending on $M$. Thus $|\hJ_{k,i}^\bft|^{-1} \mathcal{L}_{\bigcup_j J_{k,i,j}^\bft, s_{k,i}^\bft, \bft}$ is contained in a compact subset $\mathcal{K}_1$ of $L^1(\T)$. Given $m$, the intervals $\hJ_{k,i}^\bft$ with $s_{k,i}^\bft=m$ are pairwise disjoint, which implies that $\mathcal{L}_{Y_m^\bft, m, \bft}$ is contained in a compact set $\mathcal{L}_2$ of $L^1(\T)$. Since $n-m<M$, it follows that $\mathcal{L}_{Y_m^\bft, n, \bft}$ is contained in a compact set $\mathcal{L}_3$ of $L^1(\T)$. Therefore, $\varphi_n^1$ is contained in some compact set $\mathcal{K}(\eta)$ of $L^1(\T)$.
\end{proofof}

\bigskip

\bibliographystyle{alpha}

\end{document}